%% file: main.tex
\documentclass[a4paper,11pt,reqno]{amsart}
\usepackage[margin=1in]{geometry}
\usepackage[linktocpage]{hyperref}
\usepackage{amssymb, paralist, xspace, graphicx, url, amscd, euscript, mathrsfs,stmaryrd,epic,eepic,color,bm}
\usepackage{colortbl}
\usepackage{arydshln}
\usepackage[all]{xy}
\usepackage{amsthm}
\usepackage{enumerate}
\usepackage{comment}
\definecolor{kugray5}{RGB}{224,224,224}
\usepackage{lipsum}
\usepackage{multirow}
\usepackage{tikz-cd}
\usetikzlibrary{arrows,shapes,chains}
\SelectTips{cm}{}
\usepackage[T1]{fontenc}

\newtheorem{theorem}{Theorem}[section]
\newtheorem{proposition}[theorem]{Proposition}
\newtheorem{conjecture}[theorem]{Conjecture}
\newtheorem{corollary}[theorem]{Corollary}
\newtheorem{lemma}[theorem]{Lemma}

\theoremstyle{definition}
\newtheorem{definition}[theorem]{Definition}

\newtheorem{remark}[theorem]{Remark}
\newtheorem{example}[theorem]{Example}

\newtheorem{proposition-definition}[subsection]
{Proposition-Definition}
\theoremstyle{definition}
\newcommand\cE{\mathcal{E}}
\newcommand\cF{\mathcal{F}}

\newcommand\cO{\mathcal{O}}
\newcommand\cP{\mathcal{P}}

\newcommand\cX{\mathcal{X}}

\newcommand\cZ{\mathcal{Z}}

\newcommand\bF{\mathbf{F}}

\newcommand\bS{\mathbf{S}}

\newcommand\rD{\mathrm{D}}
\newcommand\rE{\mathrm{E}}

\newcommand\rH{\mathrm{H}}
\newcommand\rI{\mathrm{I}}

\newcommand\rO{\mathrm{O}}
\newcommand\rP{\mathrm{P}}

\newcommand\rR{\mathrm{R}}
\newcommand\rS{\mathrm{S}}
\newcommand\rT{\mathrm{T}}
\newcommand\rU{\mathrm{U}}

\newcommand\rW{\mathrm{W}}

\newcommand\fX{\mathfrak{X}}

\DeclareMathOperator \ch{ch}

\DeclareMathOperator \td{td}
\DeclareMathOperator \id{id}

\newcommand\CC{\mathbb{C}}
\newcommand\DD{\mathbb{D}}

\newcommand\PP{\mathbb{P}}
\newcommand\QQ{\mathbb{Q}}
\newcommand\RR{\mathbb{R}}

\newcommand\ZZ{\mathbb{Z}}

\newcommand\bfv{\mathbf{v}}

\newcommand\fG{\mathfrak{G}}

\newcommand\Aut{{\rm Aut}}
\newcommand\Bir{{\rm Bir}}

\newcommand\rank{{\rm rank}}
\newcommand\CH{{\rm CH}}
\newcommand\rk{{\rm rk}}
\newcommand\Pic{{\rm Pic}}

\newcommand\alg{{\mathrm{alg}}}
\newcommand\NS{{\mathrm{NS}}}
\newcommand\Br{{\mathrm{Br}}}
\newcommand\srM{\mathscr{M}}

\newcommand\srX{\mathscr{X}}
\newcommand\srY{\mathscr{Y}}
\newcommand\fro{\mathfrak{o}}
\newcommand\cf{{\it cf}}

\usepackage{multicol}
\title{Bloch's conjecture for (anti-)autoequivalences on  K3 surfaces}

\author{Zhiyuan Li}
\address{Zhiyuan Li, Shanghai Center for Mathematical Sciences, Fudan University, Jiangwan Campus, Shanghai, 200438, China}
\email{zhiyuan\_li@fudan.edu.cn}
\author{Xun Yu}
\address{Xun Yu,  Center for Applied Mathematics and KL-AAGDM, Tianjin University, Weijin Road 92, Tianjin 300072, P.R. China}
\email{xunyu@tju.edu.cn}
\author{Ruxuan Zhang}
\address{Ruxuan Zhang,  Beijing International Center for Mathematical Research, Peking University, No. 5 Yiheyuan Road Haidian District, Beijing, P.R.China 100871}
\email{rxzhang@pku.edu.cn}

\date{July 2023}

\begin{document}

\maketitle

\begin{abstract}

In this paper, we investigate Bloch's conjecture for autoequivalence on K3 surfaces. We introduce the notion of reflective autoequivalence of twisted K3 surfaces and  prove Bloch's conjecture for such autoequivalences, thereby confirming the conjecture for all (anti-)symplectic autoequivalences of K3 surfaces with Picard number at least $3$. The main idea is that we find a Cartan-Dieudonn\'e type decomposition of (anti)-symplectic autoequivalences. 

Our findings have several interesting consequences.
Firstly, we verify Bloch's conjecture for (anti-)symplectic birational automorphisms of Bridgeland moduli space on a K3 surface with Picard number at least $3$.  This notably implies that Bloch's conjecture holds for (anti)-symplectic birational automorphisms of finite order ($\neq 2,4$) on arbitrary  hyper-K\"ahler varieties of $K3^{[n]}$-type. Secondly, we extend Huybrechts' work in \cite{Huy12} to  twisted K3 surfaces. This extension enables us to affirm Bloch's conjecture for symplectic birational automorphisms on any hyper-Kähler variety of $K3^{[n]}$-type preserving a birational Lagrangian fibration. Finally, we prove the constant cycle property for the fixed loci of anti-symplectic involutions on hyper-Kähler varieties of $K3^{[n]}$-type, provided that $n \leq 2$ or the invariant sublattice has rank greater than $1$.
\end{abstract}

\input{sec1}
\input{sec2}

\input{sec3}

\input{sec4}

\input{sec5}

\input{sec6}

\input{sec7}

\bibliographystyle{plain}
\bibliography{main}

\end{document}

%% file: sec1.tex
\section{Introduction}

\subsection{Bloch's conjecture for (anti)-autoequivalence on $K3$ surfaces} Let $X$ be a smooth projective $K3$ surface. Let $\Phi:\rD^b(X)\to \rD^b(X)$ be a derived autoequivalence.  The goal of this paper is to study the  Bloch's conjecture for $\Phi$.   Namely, there are  natural actions
%Huybrechts has studied the action of  symplectic autoequivalences in \cite{Huy12}. Namely, let $\Phi:\rD^b(X)\to \rD^b(X)$ be a derived autoequivalence,  there is a  natural map
 \begin{equation}
     \Phi^{\CH}:\CH_0(X)_{\hom}\to \CH_0(X)_{\hom},~\hbox{and} ~\Phi^{\rm tr}:\rT(X)\to \rT(X)
 \end{equation}
induced by the Fourier-Mukai kernel of $\Phi$, where $\rT(X)\subseteq \rH^2(X,\ZZ)$ is the transcendental lattice. Then 
 \begin{conjecture}[Bloch's conjecture for autoequivalences]\label{conj2}
   For $\Phi\in \Aut(\rD^b(X))$, then $\Phi^{\rm tr}= \pm \id $  if and only if $\Phi^{\CH}=\pm \id$.  
\end{conjecture}

When  $\Phi$ is induced from a symplectic automorphism, this conjecture has been confirmed in the following cases:
\begin{itemize}
    \item  (Huybrechts \cite{Huy12}):  $\rank_2(\NS(X))\geq 4$ and $\rank_3(\NS(X))\geq 3$,
    \item (Voisin \cite{Voi12}, Huybrechts \cite{HK13}\cite{Huy12'}) $\Phi$ has finite order,
    \item (Du-Liu \cite{DL22}) $X$ is an elliptic $K3$ surface and  $\Phi$ preserves the elliptic fibration. 
\end{itemize} 
Their proofs employ entirely different techniques. In this paper, we will use a unified method to generalize their results. One of our main results is

\begin{theorem}\label{thm1}
Let $\widetilde{\NS}(X)=\ZZ\oplus \NS(X)\oplus \ZZ$ be the extended N\'eron-Severi lattice of $X$.  Let $\Phi$ be an (anti-)symplectic autoequivalence of $\rD^b(X)$.  If the action of $\Phi$ on  $\widetilde{\NS}(X)$  $$\Phi^{\widetilde{\NS}}: \widetilde{\NS}(X)\to \widetilde{\NS}(X)$$
is a product of reflective involutions, then  $\Phi^{\CH}=\mathrm{id}$ (resp. $-\mathrm{id}$). 
\end{theorem}
 Our method also applies to autoequivalence of twisted $K3$ surfaces. See Theorem \ref{thm:lagk3} for the full statement.  Here, the \textit{reflective involutions} are defined in \S \ref{subsec:ref-aut}, encompassing reflections of square $\pm 2$ vectors. The remaining problem is whether every action $\Phi^{\widetilde{\NS}}$ can be decomposed  in this way. This can be viewed as  Cartan-Dieudonn\'e type decomposition, which is a generalization of an old question of Weil \cite{W79}. We resolve this problem by extending Kneser's work \cite{Kn81} to arbitrary even lattices. This  confirms Conjecture 1.1 for arbitrary K3 surfaces with Picard number at least $3$.

%Kneser  \cite{Kn81} has studied the generators of $\widetilde{\rO}^+(\widetilde{\NS}(X))$ via using reflections around vectors of square $-2$. 
%Let $\widetilde{\rO}^+(\widetilde{\NS}(X))\subseteq \rO(\widetilde{\NS}(X) )$ be the subgroup consisting of elements which act trivially on the discriminant group of $\widetilde{\NS}(X)$ and preserve the orientation.  Thanks to the main theorem in \cite{HMS09}, the image of the  map \begin{equation}   \Aut^s(\rD^b(X)) \to \rO(\widetilde{\NS}(X))\end{equation}is  $\widetilde{\rO}^+(\widetilde{\NS}(X))$.  Let $\rW(X)\subseteq \widetilde{\rO}(\widetilde{\NS}(X)) $ be the subgroup generated by reflective involutions.  By Theorem \ref{thm1},  Conjecture \ref{conj2} is reduced to the following problem

\begin{theorem}\label{thm:thm2}
  If  $\rank~\NS(X)\geq 3$ or $X$ admits a Jacobian fibration, then   Conjecture \ref{conj2} holds. 
\end{theorem}

Let us say a few words for the case $\rank (\NS(X))\leq 2$. When $X$ has Picard number one, Conjecture \ref{conj2} is quite subtle. Indeed, there  are  infinitely many deformation types of polarized $K3$ surfaces with autoequivalences whose action on $\widetilde{\NS}(X)$ is not a product of reflective involutions. When $X$ has Picard number two, one can confirm Conjecture \ref{conj3} for $\NS(X)$ with small determinant with the aid of computer. In this situation, it remains unclear whether Theorem \ref{thm1} provided an adequate basis.

\subsection{Bloch's conjecture for hyper-K\"ahler varieties}
The first application of Theorem \ref{thm1} is to detect the Beauville-Voisin filtrations on hyper-K\"ahler varieties. 
Let $Y$ be a smooth projective hyper-K\"ahler variety of dimension $2n$ over $\CC$. Beauville and Voisin predicted that there is an increasing splitting filtration $\bF_{\bullet}^{\rm BV}\CH^\ast(Y)$, which serves as an opposite filtration to the conjectural Bloch-Beilinson (BB) filtration. The existence of such filtrations is mysterious in general.  For zero-cycles, it predicts that  there are subgroups
$\CH_{0}(Y)_{i}\subset \CH_0(Y)$ such that 
\begin{equation}\label{splitting}
    \bF^{\rm BV}_{i}\CH_0(Y)=\bigoplus_{l=0}^i\CH_{0}(Y)_{l}.
\end{equation}
 In \cite{Voi16}, Voisin  provided an increasing filtration defined by $$\bS_{i}\CH_0(Y)=\left<y\in Y|~\dim O_y\geq n-i\right>,$$ where $O_y\subset Y$ are points rationally equivalent to $y$. She has conjectured that $$\bF_{\bullet}^{\rm BV}\CH_0(Y)=\bS_{\bullet}\CH_0(Y).$$  
In recent years, there are a lot of study of Voisin's filtration when $Y$ is a Bridgeland moduli space (cf.~\cite{SYZ20, BFMS19,vial22,LZ22,lzz24}).  For instance,  there is an explicit description of the splitting $\CH_0(Y)_{\bullet}$, see Section \ref{sec:bloch-hk} for details.  

%and $\phi_*$ preserves $\bS_i$ by \cite[Corollary 3.3]{LZ22}.

%The splitting gives the BB filtration in the opposite direction.

%and there is a so-called Beauville-Voisin (BV) filtration on $\CH^\ast(Y)$  This conjectural BB filtration, denoted as $$\bF^\bullet\CH^\ast (X)\subseteq \CH^\ast (X),$$ adheres to certain natural axioms, suggesting that the graded piece $ \mathfrak{gr}^i_F\CH_0(X)$ should be dictated by the transcendental component of $\rH^{i}(X,\QQ)$. 

%To be more precisely, there should be $\CH_0(Y)_{2s}\subset \CH_0(Y)$ for $i=0,1,\ldots, n$ such that 
%\begin{equation}\label{split}    \rS_{s}\CH_0(Y)=\bigoplus_{i\leq s}^n\CH_{0}(Y)_{2i}\end{equation}
%and the filtration defined by $$\rF^{2s-1}\CH_0(Y)=\rF^{2s}\CH_0(Y)\cong \bigoplus_{i\geq s}^n\CH_{0}(Y)_{2i}$$ is conjectured to satisfy the axioms of the BB filtration. 
%The components $\CH_0(Y)_{2s}$ can be described explicitly when $Y$ is a moduli space of stable objects on some $K3$ surface. However, in general the splitting of $\rS_\bullet$ remains unknown. Axioms of the BB filtration implies that
%which is a consequence of the  Beauville-Voisin conjecture (see also  Note that any automorphism preserves the filtration $\rS_{s}\CH_0(Y)$, but for birational automorphisms, this is only known for moduli space of stable objects on a $K3$ surface  at present (c.f. \cite[Corollary 3.3]{LZ22}).

A more tractable conjecture to test the existence of such filtration is  

\begin{conjecture}[Bloch's conjecture for (anti)-symplectic birational automorphism]
\label{conj3}
 Let $Y$ be a smooth projective hyper-K\"ahler variety.  Let $\phi\in \mathrm{Bir}(Y)$  be a  birational automorphism. Then  $\phi_*$ preserves $\CH_{0}(Y)_{\bullet}$ and 
  the induced map $$\phi_i^\CH:\CH_{0}(Y)_{i}\rightarrow \CH_{0}(Y)_{i}$$ is $ (\pm 1)^i\id$ if and only if the map $\phi^\ast:\rH^0(Y, \Omega_Y^2)\to \rH^0(Y,\Omega_Y^2)$ is $\pm \id$. 

\end{conjecture}

By using the result of \cite{MZ20} and Bridgeland stability, we establish a correspondence between Conjecture \ref{conj3} and Conjecture \ref{conj2}, leading to the following result 
\begin{theorem}\label{thm:blochHK} 
If $Y$ is of $K3^{[n]}$-type,  then Conjecture \ref{conj3} holds  if one of the following conditions holds
\begin{enumerate}
\item [(1)]  $Y$ is a moduli space of stable objects on some $K3$ surface $X$ and $\rank(\NS(Y))\geq 4$.
\item [(2)] $\phi$  is symplectic and preserves a birational Lagrangian fibration on $Y$. 
\end{enumerate}
\end{theorem}

\begin{remark}
Our proof in fact gives a stronger result  under a weaker hypotheses; we refer to Theorem \ref{thm:bloch-HK-strong}  of the main text for the full statement.
\end{remark}

For finite order birational automorphisms, we have the following strengthening

\begin{theorem}[Bloch's conjecture for finite order birational automorphism]\label{thm:bloch-finite}
   Assume that $Y$ is of  $K3^{[n]}$-type and $\phi$ has finite order. Then  Conjecture \ref{conj3} holds for $\phi$ if one of the following conditions holds
   \begin{enumerate}
       \item $\mathrm{ord}(\phi)\neq 2,4$; 
      % \item  $\phi$ is an anti-symplectic birational involution and $n=2$;
       \item $\phi$ is a symplectic birational involution and its coinvariant Markman-Mukai sublattice ${\Lambda}(Y)_\phi \ncong E_8(-2)$. 
   \end{enumerate} 
\end{theorem}
A very interesting problem is to  prove the Bloch's conjecture for symplectic birational involutions with coinvariant  sublattice  $E_8(-2)$.

\begin{comment}
    Finally, we want to mention that our results  indicate Conjecture \ref{conj2} can be deduced from the generalized Franchetta conjecture on the universal family of self-product of $\Sigma$-lattice-polarized $K3$ surfaces with $\rank(\Sigma)\leq 2$.  The generalized Franchetta conjecture for lattice-polarized $K3$ surfaces and hyper-K\"ahler varieties have been studied in a series of work \cite{FLV19, Beauville2021, BL19, LZ22}.  See \S \ref{subsec:Franchetta} for more details. 
\end{comment}

 \subsection{Constant cycle Lagrangian subvariety} 
Another application is the construction of constant cycle Lagrangian subvarieties.    
Recall that a subvariety $j:Z\hookrightarrow Y$ is called a constant cycle subvariety if the image $$j_*(\CH_0(Z))\subset \CH_0(Y)$$ is constant.  If $\dim Z=\frac{1}{2}\dim Y$, then $Z$ must be Lagrangian (\cf.~ \cite[Corollary 1.2]{Voi16}), meaning that the restriction of the holomorphic $2$-form to $Z$ vanishes. Such a $Z$ is known as a constant cycle Lagrangian subvariety. In \cite{Voi16}, 
Voisin predicts that any hyper-K\"ahler variety $Y$ should contain a constant cycle Lagrangian subvariety. 

A natural source of such subvarieties is the fixed loci of an anti-symplectic involution on $Y$. In \cite{Bea11},  Beauville has proved in \cite[Lemma 1]{Bea11} that the fixed loci of such involutions are smooth Lagrangian subvarieties. Recently, the geometry of its connected component has been studied \cite{FMOS22}. Though the fixed loci itself might have a large Chow group,   Conjecture \ref{conj3} and \cite[Conjecture 2.11]{MR4370470} forces it to be a constant cycle  subvariety in $Y$. A direct consequence of Theorem \ref{thm1} is 

\begin{theorem}\label{thm:lag-cons}
Let $Y$ be a smooth projective hyper-Kähler variety of $K3^{[n]}$-type, and let $\iota\in \Bir(Y)$ be an anti-symplectic birational involution. Let $\NS(Y)^\iota\subseteq \NS(Y)$ be the invariant sublattice. Then the fixed loci of $\iota$ is a constant cycle Lagrangian  subvariety if $\rank(\NS(Y)^{\iota})\geq 2$ or $n\leq 2$.
\end{theorem}

Under the assumption of Theorem \ref{thm:lag-cons}, $Y$ has a natural associated $K3$ category.  More precisely, $Y$  is either  a Bridgeland moduli space of a twisted $K3$ surface or a double EPW sextic.  The remaining case that is not yet understood is  when $Y$ is a very general hyper-K\"ahler variety of $K3^{[n]}$-type ($n\geq 3$) with $\NS(Y)=\ZZ H$
and $H^2=2$ or $4$. 

\subsection*{Organization of the paper}
In Section \ref{sec:Bridgeland}, we review basic results of the derived category of twisted $K3$ surfaces and Bridgeland moduli spaces. We also discuss  the induced actions of derived (anti-)equivalences on Chow and cohomology groups.

In Section \ref{sec:ref}, we introduce the notion of reflective (anti-)autoequivalences of type I and II, which are automatically (anti-)symplectic autoequivalences.  Building on Huybrechts' result  \cite[Theorem 2.3]{Huy12}, we extend his finds to twisted $K3$ surfaces. Consequently, we  prove Bloch's conjecture for reflective autoequivalences. In fact, our result extend beyond this, offering  a much more general result for  autoequivalences with fixed vectors. See Theorem \ref{thm:lagk3} for details.

Section \ref{sec:bloch-hk} and \ref{sec:proof} are devoted to proving Theorem \ref{thm:blochHK} and Theorem \ref{thm:bloch-finite}.  We discuss the properties of Voisin's filtrations on zero cycles of $K3^{[n]}$-type varieties, especially on Bridgeland moduli spaces.
We establish a lifting result which shows that birational automorphisms on Bridgeland moduli spaces are induced from autoequivalences. This enables us to deduce Theorem \ref{thm:blochHK} from Theorem \ref{thm:thm2}.  Furthermore, after a  classification of all the birational automorphisms with finite orders,  we prove Theorem \ref{thm:bloch-finite}.

In Section \ref{sec:kneser}, we review the  Huybrechts involutions in the orthogonal group of Mukai lattices.  In general, they come from  the Eichler transvections and form a normal subgroup in the orthogonal group. By using  the strong approximation and congruence subgroup problem,  we obtain a Cartan-Dieudonn\'e type theorem which asserts that the stable special orthogonal group of an extended Mukai lattice is generated by reflective involutions. This implies Theorem \ref{thm:thm2}.

In Section \ref{sec:pic12}, we discuss the case of $\rank(\NS(X))\leq 2$ and provide many examples where Bloch's conjecture for autoequivalences holds. For automorphisms, we propose and discuss a reduced conjecture.

\subsection*{Acknowledgment} We  want  to  thank Ziyu Zhang, Yiran Chen and Haoyu Wu for very helpful discussions. The first author is supported by NSFC grant (No. 12121001, No. 12171090) and Shanghai Pilot Program for Basic Research (No. 21TQ00). The second author is supported by NSFC grant (No. 12071337, No. 11831013, No. 11921001).

%% file: sec2.tex
\section{Bridgeland moduli spaces}\label{sec:Bridgeland}

\subsection{Twisted $K3$ surfaces}
 Let us briefly review the basic results of twisted sheaves on $K3$ surfaces.  Our main reference is  \cite{LMS14} and \cite{HS05}.  Let $\srX\rightarrow X$ be a $\bm\mu_m$-gerbe over $X$. This corresponds to a pair $(X,\alpha)$ for some  $\alpha\in \rH^2_{fl}(X,\bm\mu_m)$, where the cohomology group is with respect to the flat topology.   The gerbe $\srX\to X$ or the  pair $(X,\alpha)$  is called a twisted $K3$ surface.  There is a Kummer exact sequence 
\begin{equation}\label{eq:Kummer}
    1\rightarrow \bm\mu_m\rightarrow {\cO_X^\times} \xrightarrow{x\mapsto x^m}  {\cO_X^\times}\rightarrow 1
\end{equation}
in flat topology and it  induces a surjective map 
\begin{equation}\label{braumap}
\rH^2_{fl}(X,\bm\mu_m)\rightarrow \mathrm{Br}(X)[m]. 
\end{equation}
We denote by $[\srX]$ the image of  $\alpha$ in $\Br(X)[m]$. 

Since $\rH^3(X,\ZZ)=0$, there is a {\it $B$-field} $B\in \rH^2(X,\QQ)$ such that $[\srX]=\delta(B)$, where  $\delta :\rH^2(X,\QQ)\to \rH^2(X, \cO_X^\times)$ is induced from the exponential exact sequence.  
Denote by  $$\widetilde \rH(X)=\bigoplus\limits_{i=0}^2\rH^{2i}(X,\ZZ)$$ the \emph{Mukai lattice} of $X$.  
The lattice structure on $\widetilde \rH(X)$ is given by the  \emph{Mukai pairing} $\left<-,-\right>$: 
\begin{equation}\label{pairing}
\left<(r, L, s), (r', L', s')\right>=L\cdot L'-rs'-r's \in \ZZ.
\end{equation}
We set $$e^B=(1,B,B^2/2)\in \widetilde{\rH}(X,\QQ),$$ and there is an isometry
$$\begin{aligned}
    {\bf exp}(B):\widetilde{\rH}(X,\QQ)&\longrightarrow \widetilde{\rH}(X,\QQ)\\
    (r,c,s) &\mapsto (r,c+rB,rB^2/2+s+c\cdot B).
\end{aligned}
$$
We can endow the Mukai lattice $\widetilde\rH(X,\ZZ)$ with a weight two Hodge structure 
determined by \begin{equation}
   \widetilde\rH^{2,0}:={\bf exp}(B)(\rH^{2,0}(X)).
\end{equation}
We call the resulting lattice the {\it twisted Mukai lattice} of $\srX$ 
and denote it by $\widetilde{\rH}(X,B,\ZZ)$ or $ \widetilde{\rH}(\srX,\mathbb{Z})$ (\cf.~\cite[Definition 2.3]{HS05}).

Up to a Hodge isometry,  $\widetilde{\rH}(\srX,\mathbb{Z})$ is independent of the choice of $B$. Usually, we will fix the choice of $B$.
The extended twisted N\'eron--Severi lattice is defined to be $$\widetilde{\NS}(\srX) := \widetilde{\rH}^{1,1}(\srX)\cap \widetilde{\rH}(\srX,\ZZ).$$ The twisted transcendental lattice is defined to be $$\rT(\srX):=\widetilde{\NS}(\srX)^\perp \subset \widetilde{\rH}(\srX,\ZZ).$$

\begin{definition}
An $n$-{\it twisted sheaf} $\cF$ on $\srX$ is an $\cO_{\srX}$-module of weight $n$ with respect to the $\mu_m$-gerbe structure ({\it cf}.~\cite[Def 2.1.2.4]{Lie07}). The Mukai vector of $\cF$ is defined as 
$$\bfv^B(\cF):={\bf exp}(B)(\ch_{\srX}(\cF)\sqrt{\td_X})\in \widetilde\NS(\srX),$$
where $\ch_{\srX}(\cF)  $ is the twisted Chern character of $\cF$ ({\it cf}.~\cite[3.3.4]{LMS14}).
\end{definition}
We will write $\rD^{(n)}(\srX)$ for the bounded derived category of coherent $n$-twisted sheaves. When $\srX=X$ is untwisted, we may write $\rD^b(X)=\rD^{(1)}(\srX).$
\subsection{Induced actions of (anti)-equivalences}
For any derived equivalence $\Phi:\rD^{(1)}(\srX)\cong \rD^{(1)}(\srX')$, it induces a Hodge isometry 
$$\Phi^{\widetilde{\rH}}:\widetilde{\rH}(\srX)\cong \widetilde{\rH}(\srX')$$ via the characteristic class of the universal object. Furthermore,  one can derive an action 
$$\Phi^\CH: \CH_0(X)_{\hom} \rightarrow \CH_0(X')_{\hom}$$
by employing the second component of the characteristic class (cf.~\cite[\S 1.2.3]{FV21})). It is easy to see that  if $E, F \in \rD^{(1)}(\srX)$  with $\bfv(E) = \bfv(F)$, then $$\Phi^\CH(c_2(E) - c_2(F)) = c_2(\Phi(E)) - c_2(\Phi(F)).$$
For anti-autoequivalences,  their actions on the  Chow group and Mukai lattices can be defined as follows. 
\begin{definition}
  If $\Phi:\rD^{(1)}(\srX)\rightarrow \rD^{(-1)}(\srX')^{\rm op}$ is a derived anti-equivalence, define $$\Phi^{\CH}:=(\DD \circ \Phi)^{\CH} \hbox{\rm ~and~} \Phi^{\widetilde\rH}:=((\DD \circ \Phi)^{\widetilde\rH})^\vee.$$
where the Hodge isometry $$(-)^\vee: \widetilde\rH(X,B,\ZZ)\to \widetilde\rH(X,-B,\ZZ)$$  is $-\id$ on $\rH^2(X,\ZZ)$ and $\id$ on $\rH^0(X,\ZZ)\oplus \rH^4(X,\ZZ)$.

  \end{definition}

A key result is 
\begin{theorem}(cf.~\cite[Corollary C]{Reinecke19})\label{Thm:lift-der}
 Let $\rO^+_{\rm Hodge}(\widetilde\rH(\srX,\ZZ))\subseteq \rO(\widetilde{\rH}(\srX,\ZZ))$ be the subgroup of orientation preserving Hodge isometries.  There is a surjection 
\begin{equation}\label{eq:t-chow}
    \Aut(\rD^{(1)}(\srX))\twoheadrightarrow \rO^+_{\rm Hodge}(\widetilde\rH(\srX,\ZZ)),
\end{equation}
which sends $\Phi \in  \Aut(\rD^{(1)}(\srX))$ to $\Phi^{\widetilde{\rH}}$, induced by the Mukai vector of the Fourier-Mukai kernel $\cE$. 
\end{theorem}

    A similar surjectivity result holds for derived anti-autoequivalences, whose actions on $\widetilde\rH(\srX,\ZZ)$ reverse the orientation.
    See Lemma \ref{lemma:anti-lift}.

\begin{definition}
   If $\Phi:\rD^{(1)}(\srX)\rightarrow \rD^{(1)}(\srX)$ is a derived (anti-)autoequivalence, we say $\Phi$ is (anti)-symplectic if   $\Phi^{\widetilde\rH}$ acts as $\pm \id$ on $\rT(\srX)$. Set $\Aut^s(\rD^{(1)}(\srX))$   to be the subgroup  of symplectic autoequivalences. We may also denote by $\Aut^s(X)\subseteq \Aut(X)$ the subgroup of symplectic automorphisms.
\end{definition}

%An immediate consequence is \begin{corollary} Let $\widetilde{\rO}(\widetilde{\NS}(\srX))\subseteq \rO(\widetilde{\NS}(\srX))$ be the stable orthogonal group of $\widetilde{\NS}(\srX)$. There is a surjection \begin{equation}    \Aut(\rD^b(\srX))\twoheadrightarrow {\rO}^+(\widetilde{\NS}(\srX))\cap \widetilde{\rO}(\widetilde{\NS}(\srX)).\end{equation}\end{corollary}

\begin{comment}
    For an object $\cE$ in  $\rD^b(\srX\times \srY)$, one can define the class $\bfv^\CH(\cE):=\ch_\srX(\cE)\sqrt{{\rm td}_{X\times Y}}\in \CH^*(X\times Y)_\QQ$.
It induces an action on the Chow group.
\begin{proposition}\cite[Corollary 2.2]{Huy19}\label{prop:isochow}
    Let $\Phi: \rD^b(\srX)\cong \rD^b(\srY)$ be a derived equivalence with Fourier-Mukai kernel $\cE$. Then the induced action 
    $$\bfv^\CH(\cE)_*:\CH^*(X)_\QQ\to \CH^*(Y)_\QQ$$ is an isomorphism.
\end{proposition}
Note that $\bfv^\CH(\cE)_*$ sends $\bfv^\CH(E)$ to $\bfv^\CH(\Phi(E))$ for any $E\in \rD^b(\srX)$.
\end{comment}
\subsection{Bridgeland stability conditions}\label{sucsec:bs}
Let us recall some results on Bridgeland stability conditions on (twisted) $K3$ surfaces. A stability condition consists of a pair $(Z,\rP)$, where $Z$ is called a central charge and $\rP$ is a slicing.

Let ${\cP}(\srX)\subseteq \widetilde{\NS}(\srX)_\CC$ be the collection of  vectors $\zeta$ such that ${\rm Re~}\zeta$ and ${\rm Im~}\zeta$ span a positive plane. Then ${\cP}(\srX)$ has two connected components. We denote by $\cP^+(\srX)$ the component which contains $e^{iH}$ for some ample line bundle $H\in \Pic(X)$ and set $$\cP_0^+(\srX)=\cP^+(\srX)\setminus \bigcup_{\delta^2=-2}\delta^\perp.$$ 
There is an open subset $U(\srX)\subset {\rm Stab(\srX)}$ consists of  geometric stability conditions in the sense of \cite{Bridgeland08}.  Denote by ${\rm Stab}^+(\srX)$  the connected component containing $U(\srX)$. Then Bridgeland showed that
\begin{theorem}\cite[Theorem 1.1]{Bridgeland08}\cite[Proposition 3.10]{HMS08}\label{deck} 
There is a covering map $$\cZ: {\rm Stab}^\dagger(\srX)\rightarrow \cP_0^+(\srX),~\sigma=(Z,\rP)\mapsto \Omega_Z.$$  Moreover, if $\srX=X$ is trivial, the subgroup of $\Aut^0(\rD^{(1)}(X))$ which preserves ${\rm Stab}^\dagger(X)$ acts as deck transformations of $\cZ.$ 
\end{theorem}
Any autoequivalence $\Psi\in \Aut(\rD^{(1)}(\srX))$ induces an action $\psi$ on $\cP_0^+(\srX)$ and hence an action on stability conditions:  $$(Z,{\rP})\mapsto (Z\circ \psi^{-1},{\rP}'),$$ where ${\rP}'(\phi)=\Psi({\rP}(\phi))$.

Let $\Aut^\dagger(\rD^{(1)}(\srX))\subset \Aut(\rD^{(1)}(\srX))$ be the subgroup of autoequivalences which respect the distinguished component ${\rm Stab}^\dagger(X)$, then \cite[Proposition 7.9]{hartmann2012cusps} (see also \cite[Theorem 2.12]{bayer2014mmp})  showed that there is a surjection:
\begin{equation}\label{Der:distinguished}
    \Aut^\dagger(\rD^{(1)}(\srX))\twoheadrightarrow \rO_{\rm Hodge}^+(\widetilde\rH(\srX,\ZZ)).
\end{equation}

\subsection{Dual stability}
In this section, we assume that $\srX=X$ is untwisted. The dual stability $\sigma^\vee$ of a stability condition $\sigma$ is defined such that an object $E$ is $\sigma$-(semi)stable with phase $\phi$ if and only if $E^\vee[2]$ is $\sigma^\vee$-(semi)stability with the opposite phase $-\phi$.(\cf.~ \cite[Proposition 2.11]{bayer2014mmp}). The derived dual functor $\DD$ induces a natural action:
\begin{equation}
    \begin{aligned}
     {\rm Stab}^\dagger(X)&\to {\rm Stab}^\dagger(X)\\  ~\sigma &\mapsto \sigma^\vee[2]
    \end{aligned}.
\end{equation}
For the central charge, we have 
$$\Re(\cZ(\sigma^\vee))=({\rm Re~}\cZ(\sigma))^\vee,~\Im(\cZ(\sigma^\vee))=-({\rm Im~}\cZ(\sigma))^\vee.$$ 
Here on the right hand sides,  $(-)^\vee:\widetilde\rH(X)\longrightarrow \widetilde\rH(X)$ is the action induced by $\DD$. 
\begin{definition}
  We define the action of a derived anti-autoequivalence $\Phi$ on ${\rm Stab}^\dagger(X)$ by $$\Phi(\sigma):=((\DD\circ \Phi)(\sigma))^\vee [2].$$ 
\end{definition}
A simple fact is 
\begin{lemma}\label{fixedsc}
    Any involution $\psi\in O^+_{\rm Hodge}(\widetilde\rH(X,\ZZ))$ can be lifted to a derived anti-autoequivalence $\Psi$, such that there exists a stability condition $\sigma$ and $\Phi(\sigma)=\sigma$. Moreover, if $\psi$ is a reflection around a vector $v$ with $v^2=2$, we can choose $\sigma$ to be $v$-generic.
\end{lemma}
\begin{proof}
    Choose an orthogonal basis of $\NS(X)\otimes \RR$ consisting of $w_1,\ldots, w_n$ and $u_1,\ldots, u_m$ such that $\psi(w_i)=w_i$ and $\psi(u_j)=-u_j$ for $i=1,\ldots ,n, ~j=1,\ldots m.$ We may assume that $w_1^2>0, u_1^2>0$ and
    $w_1+\sqrt{-1}u_1\in \cP^+(X)$. Let
    $$\Omega_Z:=\sum_{i=1}^n a_iw_i+\sqrt{-1}\sum_{j=1}^m b_ju_j,$$ with $\frac{a_i}{a_1},~\frac{b_j}{b_1}$ sufficiently small for $i=1,\ldots, n, ~j=1,\ldots ,m.$ 
    Then $\Omega_Z\in \cP_0^+(X)$ defines a central charge.
    By a straightforward computation, we obtain that $(\psi(\Omega_Z))^\vee$ equals the complex conjugation of $\Omega_Z$. 

    Therefore, Theorem \ref{deck} and the surjection \eqref{Der:distinguished} show that there is a lift $\Phi$ of $\psi\circ (-)^\vee $ such that $\Phi(\sigma)=\sigma^\vee$.
    We can obtain the first assertion by composing both sides with $\DD$.

    For the last assertion,   the central charge becomes:
    $$\Omega_Z=v+\sqrt{-1}\sum_{j=1}^m b_ju_j.$$ If there is a vector $v'$ such that $\frac{\left<\Omega_Z,v\right>}{\left<\Omega_Z,v'\right>}\in \RR$, then  $v'$ has to be proportional to $v$. This means that $\sigma$ is not lying on a $v$-wall.
    
\end{proof}

%If $\Phi$ is an anti-autoequivalence, we define the induced action on ${\rm Stab}^\dagger(X)$ by the composition of the induced actions of $\DD$ and $\DD\circ \Phi$.

\subsection{Bridgeland moduli space as HK}
Let $v\in \widetilde\rH(\srX,\ZZ)$ be an algebraic primitive vector with $v^2\geq 0$ for a twisted $K3$ surface $\srX\to X$. Consider the moduli stack $\srM_\sigma(\srX, v)$  of $\sigma$-stable sheaves on $\srX\to X$ with Mukai vector $v$. If $\sigma$ is $v$-generic, the coarse moduli space  $M_\sigma(\srX, v)$ is a smooth projective hyper-K\"ahler variety and  there exists a universal object $$\cF\in \rD^{(-1,1)}(\srX\times \srM_\sigma(\srX, v)),$$ and also  a quasi-universal object $\cE\in \rD^{(-1,0)}(\srX\times M_\sigma(\srX, v))$ in the sense of \cite[Theorem A.5]{Muk87}.
They are both unique up to a certain equivalence relation.

The following result is due to many people, for example, \cite[Sections 7 and 8]{Yos01} and \cite[Theorem 6.10, Section 7]{BM14}.
\begin{theorem}[Huybrechts-O’Grady-Yoshioka-Bayer-Macr\`\i]\label{thm:HOYBM}
Let $\srX\to X$ be a twisted $K3$ surface and let $v\in \widetilde\NS(\srX)$ be a primitive vector with $v^2\geq 0$ and $\sigma$ be a $v$-generic stability condition.  Then we have 
\begin{enumerate}
    \item $M=M_\sigma(\srX, v)$ is a smooth hyper-K\"ahler variety of $K3^{[n]}$-type with dimension $2+v^2$.  
    \item Denote by $[\_]_2$ the projection $\rH^*(M,\ZZ)\to\rH^2(M,\ZZ),$ 
    the quasi-universal object $\cE$ induces a Hodge isometry 
  \begin{equation}
  \begin{aligned}
   \theta_{\sigma,v}:v^\perp &\longrightarrow  \rH^2(M,\ZZ)\\
       u &\mapsto \frac{1}{\rho}\cdot [p_{M*}(\bfv(\cE^\vee)\cdot p_X^*(u))]_2
  \end{aligned},
  \end{equation}
if  $ v^2> 0$, while  the Hodge isometry becomes $$\theta_{\sigma,v}:v^\perp/v \cong  \rH^2(M,\ZZ),$$
if $v^2=0$. Here, $\rho$ is the similitude of $\cE$ and $\theta_{\sigma,v}$ is independent of the choice of $\cE.$
\end{enumerate}
\end{theorem}

Conversely, there is a characterization of a $K3^{[n]}$-type hyper-Kähler variety to be a moduli space of twisted objects. See \S \ref{subsection:c-Bridgeland}.

%Due to the work of Mukai \cite{Mukai84} ,  O'Grady \cite{O97}  and  Huybrechts \cite{MR1664696},   $X$ is an irreducible symplectic variety of dimension $v^2+2$ and  is of $K3^{[n]}$-deformation type. The Mukai morphism $\theta_{\sigma,v}$ is induced by a quasi-universal object, we are not sure whether the integral isometry can be induced by a universal object in $ \rD^{(-1,1)}(\srX\times \srM_\sigma(\srX, v))$. However,

%For $v^2=0$, the universal object induces a derived equivalence$\rD^b(\srX)\cong \rD^b(\srM_\sigma(\srX, v))$and a Hodge isometry $$\widetilde\rH(\srX,\ZZ)\cong \widetilde\rH(\srM_\sigma(\srX, v),\ZZ),$$ which sends $v$ to $(0,0,1)$. Let $\sigma_M$ be the $2$-form on $M_\sigma(\srX, v)$ and let $B_M$ be a B-field of $\srM_\sigma(\srX, v)$.The $\widetilde\rH^{2,0}$-part on $\widetilde\rH(\srM,\ZZ)$ is generated by $(0,\sigma_M,\sigma_M\wedge B_M)\in \widetilde\rH(\srM,\CC)$ and hence the universal object further induces a Hodge isometry $v^\perp/v\cong \rH^2(M,\ZZ)$.

\begin{comment}
    \begin{theorem}\label{moduli}
    Let $v\in\widetilde\rH(X,\ZZ)$ be a primitive Mukai vector with $v^2\geq 0$ and $\sigma\in {\rm Stab}(X)$ be a $v$-generic stability condition. Then the coarse moduli space $M_\sigma(v)$ of $\sigma-$stable objects with Mukai vector $v$ is a smooth hyper-K\"ahler variety of K3$^{[n]}$-type with dimension $2+v^2$. 
\end{theorem}
\end{comment}
%

%% file: sec3.tex
\section{Bloch's conjecture for reflective (anti-)autoequivalences}\label{sec:ref}

In this section, we introduce the notion of reflective (anti-)autoequivalences as special (anti-)symplectic autoequivalences.  We will demonstrate  Bloch's conjecture for such autoequivalences. 
\subsection{Reflective (anti-)autoequivalences}\label{subsec:ref-aut}
Throughout this section, we let $\srX\to X$ be a twisted $K3$ surface.

\begin{definition}\label{def:rde}
     Let $\Phi:\rD^{(1)}(\srX)\to \rD^{(1)}(\srX)$ (resp. $\Phi:\rD^{(1)}(\srX)\to \rD^{(-1)}(\srX)^{\rm op}$) be a derived (anti-)autoequivalence.  It is called a {\it reflective (anti-)autoequivalence} if the induced action $\Phi^{\widetilde{\rH}}$ is one of the following
  \begin{enumerate}
      \item [(I)]  $\Phi^{\widetilde{\rH}}(v)=-v$ and  $\Phi^{\widetilde{\rH}}|_{v^\perp/v}=\id$ for  some primitive vector $v\in \widetilde{\NS}(\srX)\subseteq \widetilde{\rH}(\srX)$ with $v^2= 0$.
      
      \item [(II)]  $\Phi^{\widetilde{\rH}}(v)=-v$ and $\Phi^{\widetilde{\rH}}|_{v^\perp}=\id$ for  some  primitive vector $v\in \widetilde{\NS}(\srX)\subseteq\widetilde{\rH}(\srX)$ with $v^2\neq 0$. 
  \end{enumerate}  

\end{definition}

In the case (II), as $\widetilde{\rH}(\srX)$ is unimodular and $v$ is primitive, one must have $v^2=\pm 2$.  If $v^2=-2$, $\Phi$ is a derived autoequivalence whose action is the same as spherical twist. If $v^2=2$ or $0$, $\Phi$ has to be a derived anti-autoequivalence.    In the case $v^2=2$,  this was also called a skew equivalence in \cite{HT23}. 
\begin{example}[Reflection of a prime exceptional divisor] Let us  give a geometric example of reflective autoequivalence of type (I). This is due to  \cite[Theorem 1.1]{Mark13}.  Let $Y$ be a $K3^{[n]}$-type hyper-K\"ahler variety and $E$ be an integral effective divisor with $(E,E)<0$. 
Then there exists a sequence of flops $$Y\dashrightarrow Y_1\dashrightarrow\ldots \dashrightarrow Y_n ,$$ such that the strict transform $E_n\subset Y_n$ of $E$ is contractible (\cf.~ \cite[Proposition 1.4]{D11}).
Set $e=c_1(E)$,
Markman has shown in \cite[Theorem 1.1]{Mark13} that the reflection $$s_e:\rH^2(Y,\QQ)\to \rH^2(Y,\QQ),~u\mapsto u-\frac{2(u,e)}{(e,e)}e$$ is in fact an integral monodromy operator. It can be lifted to an isometry $\widetilde{s}_e$ of $\Lambda(Y)$, sending the generator $w$ of $\rH^2(Y,\ZZ)^\perp\subset \Lambda(Y)$ to $-w$. Let $v=w+e$, then  $\widetilde{s}_e$ sends $v$ to $-v$ and is the identity  on $v^\perp/v$.
\end{example}

\begin{definition}
  An orthogonal  transformation $g\in \rO(\widetilde{\NS}(\srX))$  is called a {\it reflective involution} if  $g$ satisfies the following conditions 
 \begin{itemize}
     \item  there exists a vector $v\in \widetilde{\NS}(\srX)$ such that $g(v)=-v$ and $g|_{v^\perp}=\id$ if $v^2\neq 0$ or  $g|_{v^\perp/v}=\id$ if $v^2=0$,
     \item  $g$ acts as $\id$ on the discriminant group $A_{\widetilde{\NS}(\srX)}$.
 \end{itemize}
\end{definition}

A simple fact is 
\begin{lemma}\label{lem:lift-involution}
If $g\in \rO (\widetilde{\NS}(\srX))$ is a reflective involution,   then there exist a  reflective (anti-)autoequivalences $\Phi$ such that  $\Phi^{\widetilde{\NS}}=g$. 
\end{lemma}

\begin{proof}
As $\widetilde{\rH}(\srX)$ is unimodular,  by Nikulin's result,  the reflective involution $g$ can be extended to an integral Hodge isometry between $\widetilde{\rH}(\srX)$ satisfying the conditions in Definition \ref{def:rde}.  By \cite{HMS09}, they can be further lifted to an (anti)autoequivalence $\Psi_i$ such that $$\Phi^{\widetilde\rH}=\prod \Psi_i^{\widetilde\rH}.$$
%By  Theorem \ref{thm:aut-inj}, it suffices to verify that $$(\Psi_1\circ \Psi_2)^{\CH}=\Psi_1^\CH\circ \Psi_2^\CH, $$ for (anti)autoequivalences. When $\Psi_i$ both are autoequivalences, this is trivial. If one of them is  an anti-autoequivalence, one can compose it with the derived dual  $\DD$ and  the assertion follows from the fact  $\Psi^\CH=(\DD\circ \Psi)^\CH=(\Psi\circ  \DD)^\CH$. 
\end{proof}

In the rest of this section, we will prove Conjecture \ref{conj2} for reflective autoequivalence if either $\srX=X$ is untwisted or $\Phi$ is of type I.

%\begin{theorem}(\cite[Theorem 2.3]{Huy10})\label{thm:aut-inj}If there are two derived equivalences $\Psi_1,~\Psi_2:\rD^b(X)\rightarrow \rD^b(X')$  satisfying $\Psi_1^{\widetilde\rH}=\Psi_2^{\widetilde\rH}$, then  $\Psi_1^\CH=\Psi_2^\CH :\CH_0(X)\rightarrow \CH_0(X)$.\end{theorem}
%Recall that for a derived anti-equivalence $\Phi$, the action $\Phi^\CH=(\DD\circ \Phi)^\CH$. Clearly, the Theorem remains valid if we allow $\Psi_1,~\Psi_2$ to be derived anti-equivalences.
\subsection{Action of autoequivalences: Chow v.s. cohomology}
 Let $\Phi\in \Aut(\rD^{(1)}(\srX))$ be an autoequivalence. A pivotal discovery by Huybrechts \cite{Huy10},  is that the action $\Phi^{\CH}$ is determined by its action on $\widetilde{\rH}(\srX)$ when $\srX=X$ is untwisted. In the subsequent discussion, we extend Huybrechts' work to the twisted case and we will need the following result. 

%More generally, we consider in this section the induced action from derived equivalences between twisted $K3$ surfaces. 
    
%Though the group $\Aut^0(\rD^b(\srX))$ is still mysterious, we will show that any equivalence in this group induces the identity on the Chow group.  Roughly speaking, we first deform such equivalences to generic twisted $K3$ surfaces in the sense of \cite{Reinecke19}, which is allowed by the following result. 
\begin{proposition}(\cf.~\cite[Proposition 4.3]{Reinecke19})\label{deformtwisted}  Denote by  $\Aut^0(\rD^{(1)}(\srX))$ the kernel of $$\Aut(\rD^{(1)}(\srX))\rightarrow O^+_{\rm Hodge}(\widetilde\rH(\srX,\ZZ)).$$ For any $\Phi\in\Aut^0(\rD^{(1)}(\srX))$, 
there exist two families  of twisted $K3$ surfaces 
\begin{equation}\label{eq:fam-tk3}
    \mathfrak{X}\rightarrow T,~\mathfrak{X}'\rightarrow T,
\end{equation}
over a quasi-projective scheme $T$ satisfying the following conditions: 
\begin{enumerate}
\item there exists $t_0\in T$ so that the fibers  $\fX_{t_0}$ and $\fX'_{t_0}$ are both isomorphic to $\srX$. 
    \item  The underlying families of $K3$ surfaces are an \'etale covering of the universal family of $H$-polarized $K3$ surfaces.
    \item  for very  general $t\in T$,  the derived category $\rD^{(1)}(\fX_t) $ contains no spherical objects. 
    \item $\Phi$ can be deformed to a relative Fourier-Mukai transform $\widetilde{\Phi}$ between 
    $\rD^{(1)}(\fX/T)$ and $\rD^{(1)}(\fX'/T)$. 
\end{enumerate} 
\end{proposition}
\begin{proof}
One can take a projective family of twisted $K3$ surface 
$\fX\to T$ satisfying (1)-(3).  Here, condition (3) can be achieved  by  taking a general projective deformation of $\srX\to X$ (see the proof of \cite[Theorem A]{Reinecke19}).  As $\Phi^{\widetilde{\rH}}=\id$, it sends $(0,0,1)$ to $(0,0,1)$. Then the existence of $\fX'\to T$ and $\widetilde{\Phi}$ follows from \cite[Proposition 4.3]{Reinecke19}
\end{proof}
Then we have 
\begin{theorem}\label{thm:aut-inj}
 Let $\pi:\srX\to X$ and $\srX'\to X'$ be two twisted $K3$ surfaces.   If there are two derived equivalences $\Psi_1,~\Psi_2:\rD^{(1)}(\srX)\rightarrow \rD^{(1)}(\srX')$  satisfying $\Psi_1^{\widetilde\rH}=\Psi_2^{\widetilde\rH}$, then  $\Psi_1^\CH=\Psi_2^\CH$.
\end{theorem}

\begin{proof}
When $\srX$ and $\srX'$ are untwisted, this is exactly \cite[Theorem 2.3]{Huy10}. For twisted $K3$ surfaces,  the argument is similar.  Let us  sketch the proof here. 
Set $\Phi=\Psi_1^{-1}\circ \Psi_2 \in \Aut(\rD^{(1)}(\srX))$. Using Proposition \ref{deformtwisted}, one can obtain families  as in \eqref{eq:fam-tk3}  of twisted $K3$ surfaces  which deform $\Phi$ and satisfy the conditions (1)-(4).  Let $P\in \rD^{(-1,1)}(\srX\times \srX)$ be the Fourier-Mukai kernel of $\Phi$ and $\mathfrak{G}\in \rD^{(-1,1)}(\fX\times \fX'/T)$ be the relative  Fourier-Mukai kernel.  Let $\eta\in T$ be the generic point. By  \cite[Proposition 3.18]{HMS08},  the generic fiber $$\mathfrak{G}_\eta \in \rD^{(-1,1)}(\fX_\eta\times \fX'_\eta)$$ is a shift of some twisted sheaf.
Moreover, the specialization map to the fiber at $t_0$ produces a twisted sheaf $G=\fG_{t_0}$ over $\fX_{t_0}\times \fX'_{t_0}=\srX\times \srX$ such that $$\ch_\srX(G)=\ch_{\srX}(P)\in \CH^*(X\times X)_\QQ.$$ 

The twisted sheaf $G$ is flat over a dense open subset of $\srX\times \srX$. Since $\bfv^B(G)$ induces the identity on the twisted Mukai lattice,  the restriction $$G|_{\{\pi^{-1}(x)\}\times \srX}$$ is a twisted sheaf with Mukai vector $(0,0,1)$, and  it must support at some point  $y\in X$. 
In this manner, we define a birational map $X\dashrightarrow X$, which sends $x$ to $y$, and it extends to an automorphism $f:X\rightarrow X$.  We have $\Gamma_f\subset {\rm Supp}~ G$ is one irreducible component and the other components do not dominant $X$. Therefore $f$ is symplectic.

%we find that $\Gamma_f\subset \mathrm{Supp}(G)$ is an irreducible component and the other components do not dominate $X$. Hence 

Finally, since the deformation $\mathfrak{X}\rightarrow T$ is an \'etale covering of the universal family of $H$-polarized $K3$ surfaces, we have $f^*(H)=H$. This implies that $f$ is a finite order automorphism. Therefore $f^*$ acts as identity on $\CH_0(X)_{\hom}$ by Voisin and Huybrechts’ results.   
\end{proof}

\begin{remark}\label{rmk:anti}
Theorem \ref{thm:aut-inj} also holds for derived anti-equivalences. Indeed, if $\Phi:\rD^{(1)}(\srX)\rightarrow \rD^{(-1)}(\srX')^{\rm op}$ is a derived anti-equivalence, we can  apply Theorem \ref{thm:aut-inj} to $\DD\circ \Phi$.
\end{remark}

\subsection{Induced action on Bridgeland moduli spaces}
Let $\Phi:\rD^{(1)}(\srX)\rightarrow \rD^{(1)}(\srX)$ (resp. $\Phi:\rD^{(1)}(\srX)\rightarrow \rD^{(-1)}(\srX)^{\rm op}$) be a derived (anti-)equivalence 
. Assume that $v^2\geq 0$ is a primitive algebraic Mukai vector that $\sigma$ is a $v$-generic stability condition.  Set $w=\Phi^{\widetilde\rH}(v)$ and $\tau=\Phi(\sigma)$, where $\tau$ is a stability condition on $\rD^{(-1)}(\srX)$.
Then $\Phi$ induces an isomorphism  
\begin{equation}
\begin{aligned}
       \phi: M_\sigma(\srX, v) & \rightarrow M_{\tau}(\srX,w)\\
       E & \mapsto \Phi(E) .
\end{aligned}
\end{equation}
One should note that in the anti-equivalence case, the target is the moduli space of stable objects in $\rD^{(-1)}(\srX)$.  We still use $\srX$ just for convenience.
Then we have 

\begin{lemma}\label{lem:cohom}
Let $\phi_\ast: \rH^2(M_\sigma(\srX, v),\ZZ)\to \rH^2(M_{\tau}(\srX,w),\ZZ)$ be the induced Hodge isometry. There are commutative diagrams
    $$\begin{tikzcd}
v^\perp \arrow[rr, "\epsilon\Phi^{\widetilde\rH}"] \arrow[d, "{\theta_{\sigma,v}}"] &  & w^\perp \arrow[d, "{\theta_{\tau,w}}"] & {{\rm if~}v^2\geq 2} \\
{\rH^2(M_\sigma(\srX, v),\ZZ)} \arrow[rr, "\phi_*"]                                          &  & {\rH^2(M_\tau(\srX, w),\ZZ)}                 &                      
\end{tikzcd},$$
and $$\begin{tikzcd}
v^\perp/v \arrow[rr, "\epsilon\Phi^{\widetilde\rH}"] \arrow[d, "{\theta_{\sigma,v}}"] &  & w^\perp/w \arrow[d, "{\theta_{\tau,w}}"] & {{\rm if~}v^2=0} \\
{\rH^2(M_\sigma(\srX, v),\ZZ)} \arrow[rr, "\phi_*"]                                            &  & {\rH^2(M_\tau(\srX,w),\ZZ)}                   &                  
\end{tikzcd},$$
where $\epsilon=1$ if $\Phi$ is a derived autoequivalence and $\epsilon=-1$ if $\Phi$ is a derived anti-autoequivalence.
\end{lemma}
\begin{proof}
    When $\Phi$ is a derived autoequivalence, this  is  essentially \cite[Proposition 3.5]{Ouchi18}.
   If $\Phi$ is a derived anti-autoequivalence,  as $\DD\circ \Phi\in \Aut(\rD^{(1)}(\srX))$,  it suffices to deal with the case $\Phi=\DD$ and show that $\epsilon=-1.$

%The action of $\DD$ on $\widetilde{\rH}(X)$ is  $$(r,c,s)^\vee:=\DD^{\widetilde{\rH}}(r,c,s)= (r,-c,s).$$ 

For simplicity, we write  $M=M_\sigma(\srX, v)$, $N=M_\tau(\srX, w)$ and assume that there is a universal object $\cE$ on $\srX\times M$.
Let $p:\srX\times M\to M$ and $q:\srX\times M\to \srX$ be the natural  projections. Recall that the map $\theta_{\sigma,v}$ is induced by $\cE^\vee$ and defined by the composition
\begin{equation}
    v^\perp \hookrightarrow \widetilde\rH(\srX) \xrightarrow{\Phi_{\cE^\vee}}  \rH^*(M)  \xrightarrow{[\_]_2} \rH^2(M),
\end{equation}
which sends $u$ to $[p_*(\bfv(\cE^\vee)\cdot q^*u)]_2$.
One observes that $[p_*(\bfv(\cE^\vee)\cdot q^*u)]_2$ only depends on the K\"unneth component $$[\bfv(\cE^\vee)\cdot q^*u]_{\rH^2(M)\otimes \rH^4(\srX)}=\bigoplus_{0\leq j\leq 2}[\bfv(\cE^\vee)]_{\rH^2(M)\otimes \rH^{4-2j}(X)}\cdot [q^*u]_{\rH^0(M)\otimes \rH^{2j}(\srX)},$$ where $[\_]_{\rH^i(M)\otimes \rH^j(\srX)}$ is the projection $\rH^*(M\times \srX)\to \rH^i(M)\otimes \rH^j(\srX).$

On the other hand, the map $\theta_{\tau,w}$ is induced by $\phi_*(\cE^{\vee})^\vee$. Since $\phi$ is an isomorphism, we find that $\phi_*(\cE^{\vee})^\vee=\phi_*(\cE)$.
By a straightforward computation, we obtain $$[\bfv(\cE)]_{\rH^2(M)\otimes \rH^{4-2j}(\srX)}=(-1)^{j+1}[\bfv(\cE^\vee)]_{\rH^2(M)\otimes \rH^{4-2j}(\srX)},$$for $j=0,1,2$.  Then one can observe the following equality:
$$[\bfv(\cE^{\vee\vee})\cdot q^*(-u^\vee)]_{\rH^2(M)\otimes \rH^{4-2j}(\srX)}=[\bfv(\cE^\vee)\cdot q^*u]_{\rH^2(M)\otimes \rH^{4-2j}(\srX)},$$which implies that $\theta_{\tau,w}(-u^\vee)=\phi_*\theta_{\sigma,v}(u)$. Since $\Phi^{\widetilde\rH}=(\_)^\vee$, we obtain that $\epsilon=-1$ for $\Phi=\DD.$
\end{proof}

\subsection{Bloch's conjecture for reflective autoequivalence of type I}
In fact, we have a much stronger result.
\begin{theorem}\label{thm:lagk3}
Let $\Phi:\rD^{(1)}(\srX)\rightarrow \rD^{(1)}(\srX)$ (resp.~$\Phi:\rD^{(1)}(\srX)\rightarrow \rD^{(-1)}(\srX)^{\rm op}$)  be a derived (anti-)autoequivalence. Assume that the action on the transcendental lattice is $\pm \id$.
Then $\Phi^{\CH}=\pm\id$ (resp. $\mp \id$) if  $\Phi^{\widetilde \rH}(v)=\pm v$ for some isotropic vector $v\in \widetilde{\NS}(\srX)$ and $\Phi^{\widetilde \rH}(u)=\pm u$ for some $u\in v^{\perp}/v$ with $u^2\geq 0$.
\end{theorem}
\begin{proof}
Let $\srX'\to X'$ denote the moduli stack  $\srM_\sigma(\srX, v)\to M_\sigma(\srX,v)$ of $\sigma$-stable objects with Mukai vector $v$, where $\sigma$ is a $v$-generic stability condition.  Since $v^2=0$, $X'$ is also a $K3$ surface.
 Note that the Fourier-Mukai transform $\rD^{(1)}(\srX)\to \rD^{(1)}(\srX')$  induces an autoequivalence 
\begin{equation}
    \Psi:\rD^{(1)}(\srX')\to \rD^{(1)}(\srX')
\end{equation}
on $\srX'$.  There is a commutative diagram
\begin{equation*}
    \begin{tikzcd}
\CH_0(X)_{\hom}\arrow[rr, "\Phi^{\CH}"] \arrow[d, "{\cong }"] &  & \CH_0(X)_{\hom}  \arrow[d, "{\cong}"]\\
\CH_0(X')_{\hom} \arrow[rr, "\Psi^{\CH}"]                                         &  & \CH_0(X')_{\hom}                                 
\end{tikzcd},
\end{equation*}
 and it suffices to show that $\Psi^{\CH}=\pm \id$ on $\CH_0(X')_{\hom}$. 
 
 For simplicity, we may assume $\Phi^{\widetilde{\rH}}(v)=v$ otherwise we can replace $\Phi$ by $\Phi[1]$. From our construction, the Hodge isometry $\Psi^{\widetilde{\rH}}:\widetilde{\rH}(\srX',\ZZ)\to \widetilde{\rH}(\srX',\ZZ)$ sends $(0,0,1)$ to $(0,0,1)$ and hence $\Psi^{\widetilde{\rH}}(1,0,0)$ must be of the form $e^l=(1, l, \frac{l^2}{2})$ for some $l\in \rH^2(X',\ZZ)$.  Let us identify $ v^\perp /v$ with $ \rH^2(X',\ZZ)$. 
\vspace{.2cm}

\noindent \textbf{Claim A}: there exists a derived  (anti-)equivalence $$\widehat{\Psi}:\rD^{(1)}(\srX')\to \rD^{(1)}(\srX')~(\hbox{or}~\rD^{(1)}(\srX')\to\rD^{(-1)}(\srX')^{\rm op}),$$ 
such that
\begin{itemize}
    \item $\widehat{\Psi}^{\widetilde{\rH}}=\id_{\rU}\oplus \widehat{\Psi}^{\rH^2}$ is diagonal. 
    \item $\widehat{\Psi}^{\CH}=\Psi^{\CH}$ (resp.~$-\Psi^{\CH}$).
    \item $\widehat{\Psi}^{\rH^2}(u)=u$. 
\end{itemize}
\vspace{.1cm}

    We first treat the case that $\Psi:\rD^{(1)}(\srX)\to \rD^{(1)}(\srX)$ is an autoequivalence.   Fix a $B$-field $B$ of $\srX'$, we may also use the notation $\widetilde{\rH}(X',B,\ZZ)=\widetilde{\rH}(\srX')$ to keep track the change of $B$-fields. Let $\sigma_{X'}\in \rH^{2,0}(X')$ be the holomorphic $2$-form.  Then $\rH^{2,0}(\srX')$ is spanned by  $\sigma_{X'}+B\wedge \sigma_{X'}$.     If $\Psi$ is symplectic, we have
\begin{center}
    $\Psi^{\widetilde{\rH}}(\sigma_{X'}+B\wedge\sigma_{X'})=\sigma_{X'}+B\wedge\sigma_{X'}$ and $\Phi^{\widetilde{\rH}}(B\wedge\sigma_{X'})=B\wedge\sigma_{{X'}}$.
\end{center}
 Therefore $\Phi^{\widetilde{\rH}}(\sigma_{X'})=\sigma_{X'}$ and $\ell\in \sigma_{X'}^\perp$.  This means $l=c_1(L)\in\NS(X')$  for some $L\in\Pic(X')$. One can take 
$$\widehat{\Psi}: \rD^{(1)}(\srX')\xrightarrow{\Psi}  \rD^{(1)}(\srX') \xrightarrow{L^{-1}\otimes (\underline{~~})} \rD^{(1)}(\srX'). $$ It is as desired. 

If $\Phi$ is anti-symplectic, we have $l+2B\in \NS(X')_\QQ$ by a similar computation. Since $l$ is integral, we obtain 
$$2[\srX']=0\in \Br(X').$$
We may assume that $\srX'$ is a $\mu_2$-gerbe and $2B\in \rH^2(X',\ZZ)$, otherwise we can replace $\srX'$ by a $\mu_2$-gerbe $\srX''\to X'$ such that there is an equivalence $\rD^{(1)}(\srX'')\cong \rD^{(1)}(\srX')$ whose action on the Mukai lattice preserves  $(0,0,1)$.  Then  we have $\rD^{(1)}(\srX')=\rD^{(-1)}(\srX')$ and the class
$l+2B=c_1(L)\in \NS(X')$ is integral. 
We can take the composition 
$$\widehat{\Psi}: \rD^{(1)}(\srX')\xrightarrow{\Psi}  \rD^{(1)}(\srX') \xrightarrow{L^{-1}\otimes (\underline{~~})} \rD^{(1)}(\srX'),$$  which induces  a diagonal transform $$\id_{\rU}\oplus \widehat{\Psi}^{\rH^2}:\widetilde\rH(X',B,\ZZ)\to \widetilde\rH(X',B,\ZZ).$$
It is clear that $\Psi^{\CH}=\widehat{\Psi}^{\CH}$ as they differ by tensoring a line bundle.  
 \vspace{.1cm}

Now we  deal with the case  $\Psi$ is an anti-autoequivalence. A natural idea is to compose  $\Psi$ with the derived dual $\DD$ and reduce  it to the case of auoequivalences. In our situation,  suppose $\Psi$ is symplectic, a similar computation 
shows that $$l-2B\in \NS(X')_\QQ.$$
As before, we may assume that $\srX'\to X'$ is a $\mu_2$-gerbe. Then the derived dual functor $$\DD:\rD^{(1)}(\srX')\to  \rD^{(-1)}(\srX')^{\rm op}$$ is an anti-autoequivalence and we can return to the autoequivalence case by composing $\DD$.  %As $\DD\circ \Psi[1]$ is an anti-symplectic autoequivalence with the action on $\widetilde\rH(\srY)$ as $\id_{\rU}\oplus \Phi^{\rH^2}$, we have $(\Psi)^{\CH}=-(\DD\circ \Psi[1])^{\CH}=-\id$ by Claim B. 

When $\Psi$ is  anti-symplectic, then we have $l=c_1(L)\in \NS(X')$. 
We can take the composition 
$$\widehat{\Psi}: \rD^{(1)}(\srX')\xrightarrow{\Psi}  \rD^{(1)}(\srX')^{\rm op} \xrightarrow{L^{-1}\otimes (\underline{~~})} \rD^{(1)}(\srX')^{\rm op}\xrightarrow{\DD[1]}  \rD^{(-1)}(\srX'),$$ which
induces  a diagonal map $\id_{\rU}\oplus \widehat{\Psi}^{\rH^2} :\widetilde\rH(X',B,\ZZ)\to \widetilde\rH(X',-B,\ZZ).$ In this case, we have  $\Psi^\CH=-\widehat{\Psi}^\CH.$ 

At last,  we have $\widehat{\Psi}^{\rH^2}(u)=\pm u$.   We can compose it with the derived dual functor $\DD$ and obtain the desired $\widehat{\Psi}$.   This finishes the proof of Claim A.
\vspace{.1cm}
 
Now, since $\Psi^{\CH}=\pm \widehat{\Psi}^{\CH}$, we can replace $\Psi$ by $\widehat{\Psi}^{\CH}$ and   reduce to consider the case $\Psi^{\widetilde{\rH}}$ is diagonal. 
The assertion  follows from the claim below.  
\vspace{.1cm}

\noindent \textbf{Claim B: there exists a (anti-)symplectic automorphism $\varphi\in \Aut(X')$ such that   $\varphi^{\CH}=\Psi^{\CH}=\pm \id$.}
~\vspace{.1cm}

Recall that under our assumption,  $\Phi^{\widetilde \rH}$ fixes a nonzero vector  $u\in  v^{\perp}/v$ with $u^2\geq 0$.
 By Theorem \ref{thm:HOYBM}, we can regard $u$ as a primitive effective class in $\NS(X')$.
We may further assume that $u$ is a nef class on $X'$. If not, there exist a sequence of spherical twists $T_{C_1}, \ldots ,T_{C_m}$ with $C_i\in \NS(X')$ such that $u'=T_{C_1}\circ \ldots \circ T_{C_m}(u)$ is nef in $\NS(X')$ (\cf.~\cite[Corollary 3.2.9]{Huy16}). Note that $T_{C_i}$ can be naturally lifted  to a spherical twist for $\rD^{(1)}(\srX')$.  So we can  replace $\Psi$ by $$T_{C_1}\circ \ldots \circ T_{C_m}\circ \Psi\circ (T_{C_1}\circ \ldots \circ T_{C_m} )^{-1},$$
which preserves the nef class $u'$ and does not change the action on $\CH_0(X')_{\hom}$.  We can prove the claim by cases.

$\bullet$ In the case $u^2>0$, if $u$ is already ample, the Hodge isometry $\Psi^{\rH^2}:\rH^2(X',\ZZ)\to \rH^2(X',\ZZ)$ given by 
    \begin{equation}\label{eq:H2}
        \rH^2(X',\ZZ) \cong v^\perp/v\xrightarrow{\Phi^{\widetilde{\rH}}} v^\perp/v\cong \rH^2(X',\ZZ),
    \end{equation}
can be lifted to an automorphism $\varphi\in\Aut(X')$ of finite order. Then $\varphi^{\CH}=\pm \id $  by the main results of Huybrechts and Voisin in \cite{Voi12, Huy12'}.  
Note that if an automorphism $\varphi\in \Aut(X')$ acts on $\rH^{2,0}(X')$ as $\pm \id$, it can be naturally  lifted to a derived (anti-)equivalence $$\varphi_*:\rD^{(1)}(\srX')\to \rD^{(\pm 1)}(\srX')$$ for any $\srX'\to X'$. 
By Theorem \ref{thm:aut-inj} and Remark \ref{rmk:anti}, we have $\varphi^{\CH}=\Psi^{\widetilde{\CH}}$ as well.

% Therefore $\Psi^\CH=\pm \id$ by Lemma \ref{aux}.

%Here, we view the Fourier-Mukai kernel $${\rm cone}(\DD(\cO_{C_i}(-1))\boxtimes \cO_{C_i}(-1)\rightarrow \Delta)$$ as objects in $\rD^{(-1,1)}(\srY\times \srY)$ and we also have $T_{C_i}^\CH=\id$.
%as the derived dual functor $$\DD:\rD^b(\srY)\to \rD^{(-1)}(\srY)^{\rm op}$$ is not an anti-autoequivalence, we can not apply Theorem \ref{cohomtochow} directly to $\DD\circ \Phi'$. we can reduce to the autoequivalence case as the following.

If $u$ is big and nef but not ample,  let $C$ be the exceptional curve (i.e. ($u\cdot C)=0$) and let $C_1,\ldots, C_n$ be the irreducible components of $C$.  Let $h\in\NS(X')$ be an ample class. If $\Psi^{\rH^2}(h)$ is not ample, then there exists $1\le i\le n$ such that $(\Psi^{\rH^2}(h),C_i)<0$. Then we may replace $\Psi$ by $T_{C_i}\circ \Psi$. Such process must terminate after finitely many steps since the integer $(h,\Psi^{\rH^2}(h))$ is positive and decreasing. Thus, one can find an automorphism $\varphi$ which is a lift of $\Psi^{\rH^2}$ (up to composing spherical twist) and it will fix the big and nef class $u$. Such automorphism is of finite order and the assertion follows similarly.

$\bullet$  In the case $u^2=0$,   $u$  is base point free and defines an elliptic fibration $X'\to \PP^1$ (\cf.~\cite[Proposition 3.10]{Huy16}).   Similar to the previous case, $\Psi^{\rH^2}$ (up to composing some spherical twists) can be lifted to an automorphism $\varphi$ which preserves the elliptic fibration.  The assertion then follows from the main result of \cite{DL22}. See also Remark \ref{rmk:ant-ell}.  This proves Claim B. 
\vspace{.05cm}

 \end{proof}

\begin{remark}\label{rmk:ant-ell}
 In \cite{DL22}, Du and Liu mainly deal with the case $\varphi$ is symplectic and preserves an elliptic fibration $X'\to \PP^1$, but their argument can still apply to the anti-symplectic case. Roughly speaking,  let $J(X')\to \PP^1$ be the relative Jacobian of $X'\to \PP^1$,  the induced automorphism  $\varphi':J\to J$ is of finite order by \cite[Lemma 4.6 \& Corollary 2.3]{DL22}.
Hence  $(\varphi')^{\CH}=-\id$. It follows from \cite[Lemma 4.3 \& 4.4]{DL22} that $\varphi^{\CH}=-\id$. 
\end{remark}

%Let $\Phi$ be a derived autoequivalence  of $X$ and let  $v\in\widetilde\rH(X,\ZZ)$ be a primitive vector fixed by $\Phi^{\widetilde\rH}$ with $v^2\geq 0$.Let $\sigma$ be a $v$-generic stability condition  and let $M=M_\sigma(X,v)$ be the moduli space of $\sigma$-stable objects in $\rD^b(X)$. 
 
% such that $\tau=\Phi(\sigma)$ and $\sigma$ are in the same chamber with respect to $v$. Hence $M_\sigma(X,v)$ is naturally identified with $M_\tau(X,v)$, and 

\subsection{Bloch's conjecture for reflective autoequivalence of type II}From now on, we may assume that $\srX=X$ is untwisted.

\begin{lemma}\label{modulitok3}
If $E \mapsto \Phi(E)$ induces an automorphism 
    \begin{equation*}
            f:M\to  M,
    \end{equation*}  so that $f_\ast$ acts trivially on $\CH_0(M)$, then $\Phi^\CH$ is the identity. 
\end{lemma}
\begin{proof}
By our assumption, we have  $c_2(E)=c_2(\Phi(E))$ according to Theorem \ref{thm:K3HK}.
    It suffices to show that the Chern class map
    \begin{equation}\label{eq:chernM}
        \CH_0(M)_{\hom}\rightarrow\CH_0(X)_{\hom}
    \end{equation} is surjective. 
    In fact, the surjectivity of \eqref{eq:chernM} follows from the obvious surjectivity for $X^{[n]}$ and \cite[Section 2.2]{SYZ20}.
\end{proof}

Then we have 
\begin{theorem}\label{thm:maingeo}Assume that $\srX=X$ is untwisted.
 If $\Phi$ is a reflective (anti-)autoequivalence, the map $$\Phi^{\CH}:\CH_0(X)_{\hom}\to \CH_0(X)_{\hom}, $$ is $\pm \id$. 
\end{theorem}
\begin{proof}
  By assumption,  there is $v\in \widetilde{\NS}(X)$ satisfying the conditions in Definition \ref{def:rde}.  We can divide the proof into 3 cases. 
    
  1)  If  $v^2=-2$, the reflection can be lifted to a spherical twist, which acts on $\CH_0(X)$ as identity. 
    
   2)  If $v^2=2$, we apply Lemma \ref{fixedsc} to $\Phi[1]$. So there exists another derived anti-autoequivalence $\Psi$, such that $\Psi^{\widetilde{\rH}}=-\Phi^{\widetilde{\rH}}$ and there exists a $v$-generic stability condition fixed by $\Psi$. The derived anti-autoequivalence $\Psi$ induces an automorphism:
    $f\in \Aut(M_\sigma(X, v))$ with $$f_*:\rH^2(M_\sigma(X, v),\ZZ)\rightarrow \rH^2(M_\sigma(X, v),\ZZ)$$ equaling identity by Lemma \ref{lem:cohom}.
    Since $M_\sigma(X, v)$ is of $K3^{[n]}$-type, we obtain that $f=\id.$
    Therefore, we have  $$\Phi^{\CH}=-\Psi^{\CH}=-\id,$$ by Lemma \ref{modulitok3} and Theorem \ref{thm:aut-inj}.

    3) if $v^2=0,$ there exists $w\in v^\perp/v$ such that $\Phi^{\widetilde{\rH}}(w)=w$ with $w^2>0$.  
    Then the assertion follows from  Theorem \ref{thm:lagk3}.
\end{proof}

Theorem \ref{thm1}  is an immediate consequence of  Lemma \ref{lem:lift-involution}, Theorem \ref{thm:aut-inj} and Theorem \ref{thm:maingeo}.  
\vspace{.1cm}

For Theorem \ref{thm:thm2}, as we will show later, when $\rank~\NS(X)\geq 3$ or $\NS(X)$ contains a hyperbolic plane $\rU$,  any $\Phi^{\widetilde{\NS}}$ can be decomposed into  a  product of  reflective involutions.  This will be proved in \S 6 and \S7.

\subsection{An alternative proof via moduli space of vector bundles} 
For reflections of square $2$ vectors,   there is another approach by using only Giesker stability conditions. 
Let $T_{\cO_X}:\rD^b(X)\to \rD^b(X)$ be the spherical twist around $\cO_X$ and $v=(r,D,s)$ be a Mukai vector. Let $d=\frac{v^2}{2}+1$ and $H$ be an ample line bundle. 
In \cite{Yoshi09}, Yoshioka has shown that if we assume 
\begin{equation}\label{eq:larged}
    D\cdot H>\max\{ 4r^2+1, 2r(v^2 H^2-D^2 H^2)\},  ~s>2d,  \hbox{~if $r>0$},
\end{equation}
or \begin{equation}\label{eq:larges}
    s>\max\left\{2d, \frac{(D\cdot H)^2}{2H^2}+1\right\}, \hbox{~if $r=0$},
\end{equation}
then $T_{\cO_X}[1]$ induces an isomorphism 
$$ \varphi: M_H(v)\to M_{H}(\widehat{v}),$$
such that for any stable sheaf $E\in M$, $\varphi(E)=T_{\cO_X}(E)[1]$ is a $\mu_H$-stable vector bundle with Mukai vector $\widehat{v}=(s,-D,r)$. 
In fact, any $\mu_H$-stable sheaf $E$ with $\rank(E)\geq \frac{\bfv(E)^2}{2}+2$  is a vector bundle. 

Next, we have the following result due to Huybrechts.
\begin{proposition}(\cf.~\cite[Proposition 6.1]{Huy08})\label{prop:aut-stable}
    Let $X'$ be a fine moduli space of $\mu_H$-stable isotropic vector bundles on $X$ and let $\Phi:\rD^b(X)\to \rD^b(X')$ be the corresponding Fourier-Mukai transform.  Then there exists a polarization $H'$ on $X'$, such that for any $\mu_H$-stable vector bundle $E$ with $\bfv(E)^2\geq 2$, we have $$\Phi(E)\cong F[-1],$$ where $F$ is a $\mu_{H'}$-stable vector bundle on $X'$.
\end{proposition}
For a derived equivalence $\Phi:\rD^b(X)\to \rD^b(X')$, $X'$ may not be a moduli space of $\mu_H$-stable vector bundle.
There is a technical issue that we need to make sure $\Phi^{-1}(k(x))$  is  a shift of a vector bundle for any skyscraper sheaf $k(x)$. To overcome this problem, we can use the spherical twist $T_{\cO_X}$ and tensor line bundles to suitably modify $\Phi$. This is valid  when $\rank (\Pic(X))=1$.
\begin{lemma} \label{DVB}
   Assume $\Pic(X)\cong \ZZ H$ and let $\Psi:\rD^b(X)\to \rD^b(X)^{\rm op}$ be a derived anti-autoequivalence such that $\Psi^{\widetilde{\rH}}$ is a $(+2)$-reflection $s_w$. Then there is a derived anti-autoequivalence $\Phi$ such that 
   \begin{itemize}
      \item $\Psi^\CH=\Phi^{\CH}$,
       \item $(\DD\circ \Phi)^{-1}(k(x))$ is a shift of a vector bundle for any $x\in X$,
       \item $\Phi^{\widetilde{\rH}}$ is also a $(+2)$-reflection $s_{w'}$, such that $M_H(w')$ consists of vector bundle.
   \end{itemize}
\end{lemma}
\begin{proof}
Let $H_n=-\otimes H^{\otimes n}$ be the derived equivalence sending $E$ to $E\otimes H^{\otimes n}.$
We know that both ${T}_{\cO_X}$ and $H_n$ act as identity on $\CH_0(X)_{\hom}$. Let
    $$\Phi:=  {T}_{\cO_X}  \circ H_m\circ {T}_{\cO_X}  \circ H_n\circ \Psi  \circ H_{ -n}\circ {T}_{\cO_X}\circ H_{ -m}\circ {T}_{\cO_X}.$$ 
By a direct computation, we have the following commutative diagram
    \begin{equation}
        \begin{tikzcd}
\widetilde \rH(X) \arrow[rrr, "{T}_{\cO_X}  \circ H_m\circ{T}_{\cO_X} \circ H_n"] \arrow[d, "s_w"] &  &  &\widetilde \rH(X) \arrow[d, "s_{w'}"] \\
\widetilde \rH(X) \arrow[rrr, "{T}_{\cO_X}  \circ H_m\circ{T}_{\cO_X} \circ H_n"]                  &  &  &\widetilde \rH(X)    ,           
\end{tikzcd}
    \end{equation}
    where $w'= {T}_{\cO_X}  \circ H_m\circ {T}_{\cO_X} \circ H_n(w)$.
Note that we may assume that $w'=(r,c,s)$ such that $r,c,s$ are all sufficiently large and $r\gg s.$
In particular, $M_H(w')$ consists of $\mu_H$-stable vector bundle (up to shift) by \eqref{eq:larged}.

Write $(\DD\circ \Phi)^{-1}=T_{\cO_X}\circ \Theta$, then we compute that $$\bfv(\Theta(k(x)))=-(rs+1,-cr,r^2).$$ 
Consider the moduli space of $H$-stable sheaves with Mukai vector $(rs+1,-cr,r^2)$. We obtain a derived equivalence $\Theta'$ inducing the same action on $\widetilde\rH(X,\ZZ)$ as $\Theta$ such that $\Theta'(k(x))$ is a $H$-stable sheaf.

Again, Yoshioka's result shows that $T_{\cO_X}\circ \Theta'(k(x))$ is a shift of a $\mu_H$-stable vector bundle and due to Theorem \ref{thm:aut-inj} the anti-equivalence $\DD\circ \Theta'^{-1}\circ T_{\cO_X}^{-1}$ satisfies the conditions in the statement.
\end{proof}
Then we can apply Proposition \ref{prop:aut-stable} and the same argument as Theorem \ref{thm:maingeo} implies $$\Phi^{\CH}=-\id,$$ when $\Pic(X)=\ZZ H$. By taking specializations,  one can get $\Phi^{\CH}=-\id$ for any $X$.

%% file: sec4.tex
\section{Birational automorphism on Bridgeland moduli spaces}\label{sec:bloch-hk}

In this section, we review  Voisin's filtration on  Bridgeland moduli spaces and show that the action of birational automorphism on Voisin's filtration is determined on the first graded piece. Moreover, we classify the hyper-K\"ahler varieties with finite order birational automorphisms and show that most of them are Bridgeland moduli spaces. 

\subsection{Filtrations on zero cycles of Bridgeland moduli spaces}
 If $M$ is a moduli space of stable objects on a $K3$ surface $X$, Voisin's filtration on $\CH_0(M)$ is controlled by a canonical filtration on $X$. More precisely, let $\fro_X$ be the BV class on $X$.  O'Grady  \cite{OG13} introduced a set filtration $\bS_\bullet(X)$  on $\CH_0(X)$
$$ \bS_i(X) =\bigcup_{\deg([z])=i}\{ [z] + \ZZ \cdot [\fro_X] \},$$ where z runs through all effective zero-cycles $z$ of degree $i$. 
There are several constructions of the BV filtration on $\CH_0(M)$ (\cf.~\cite{LZ22}) and all of them are proved to be equivalent. In particular, we have 
\begin{theorem}\cite[Theorem 1.1]{LZ22}\label{thm:SYZ=BV}
   $\bS_i\CH_0(M)= \left<E \in M |~c_2(E)\in \bS_i(X)\right>$. 
\end{theorem}
Next, for a Hilbert scheme $X^{[n]}$, there are components
$$\CH_{0}(X^{[n]})_{i}=\left<(o,\ldots,o,x_1-o,\ldots ,x_{i}-o)~|~x_1,\ldots, x_i\in X\right>\subset \CH_0(X^{[n]}),$$
which gives the splitting of Voisin's filtration (\cf.~\cite[Theorem 2.5]{Voi12}) in the sense that
\begin{equation}
    \bS_{i}\CH_0(X^{[n]})=\bigoplus_{l=0}^i\CH_{0}(X^{[n]})_{l}.
\end{equation}
For moduli space of stable objects in $\rD^b(X)$, Shen-Yin-Zhao and Vial have shown:
\begin{theorem}\cite[Theorem 3.1]{vial22}\cite[\S 2.2]{SYZ20} \label{bmotive}
Let $M$ be a moduli space of stable objects in $\rD^b(X)$ of dimension $2n$. Then there is a component $R_0$ of the incidence variety $$R= \{(E,\xi)~|~c_2(E)=[\mathrm{supp}(\xi)]+c\cdot \fro_X\in \CH_0(X) \}\subset M\times X^{[n]}$$ dominants both factors and it induces
an isomorphism of the birational motives of $M$ on $X^{[n]}$  as co-algebra objects.
%a component of the incidence variety $$R= \{(E,\xi)~|~c_2(E)=[\mathrm{supp}(\xi)]+c\cdot \fro_X\in \CH_0(X) \}\subset M\times X^{[n]}$$ dominants both factors. Moreover, it induces an isomorphism of the birational motive of $M$ on $X^{[n]}$  as co-algebra objects.
\end{theorem} 
In particular, Voisin's filtration on $\CH_0(M)$ also splits $$\bS_i\CH_0(M)=\bigoplus_{l=0}^i\CH_{0}(M)_{l},$$ which is determined by the isomorphism induced by $R_0$.  The first piece is spanned by $[E_0]$ with $c_2(E_0)=k\fro_X$. 

If $\phi:M\dashrightarrow M$ is a birational automorphism, then $\phi_*$ preserves the BV filtration, (\cf.~\cite[Corollary 3.3]{LZ22}). Therefore, $\phi$ induces $$\phi^\CH_i:\CH_{0}(M)_{i}\rightarrow \CH_{0}(M)_{i}.$$

   % If $\srX\to X$ is a twisted K3 surface and $M=M_\sigma(\srX,v)$,  one may extend O'Grady's filtration on $\CH_0(X)_\QQ$ and similarly define a filtration 
   % $$\rS^{\rm SYZ}_\bullet\CH_0(M):=\left<E \in M |~c_2(E)\in \rS_i(X)\right>.$$ However, it is unknown whether Theorem \ref{thm:SYZ=BV} still holds.  Even the termination $\rS^{\rm SYZ}_{n}\CH_0(M)=\CH_0(M)$ is not known yet.

%\begin{conjecture}\label{conj:3'}Let $Y$ be a $K3^{[n]}$-type hyper-K\"ahler variety. If  $\phi \in \Aut(Y)$ is  a symplectic automorphism. Then $$\phi_*:\CH_0(Y)\rightarrow \CH_0(Y)$$ is $\id$. Let  $M=M_\sigma(v)$ be a moduli space of stable objects in $\rD^b(Y)$. If  $\phi \in \Aut(M)$ is  an anti-symplectic automorphism. Then $$\phi^{\CH}_{2s} :\CH_0(M)_{0,2s}\to \CH_0(M)_{0,2s}$$ is $(-1)^s \id.$\end{conjecture}

\subsection{A variant of Marian-Zhao} 
Marian and Zhao have demonstrated in \cite{MZ20} that two sheaves in $M$ are rationally equivalent if and only if their second Chern classes coincide in $\CH_0(X)$. They presented this result for untwisted moduli spaces, yet it is valid for twisted ones as well,  with a minor alteration. Here we give a variation of the Marian-Zhao Theorem, which will be used later.

\begin{theorem}[cf.~\cite{MZ20, SYZ20}]\label{thm:K3HK}
Let $\srX\to X$ be a twisted $K3$ surface. Let $M$ be a moduli space of stable  objects in $\rD^{(1)}(\srX)$. Then we have 
\begin{enumerate}
    \item[(i)]   For any $E,F\in M$,   $[E]=[F]$ in $\CH_0(M)$ if and only if $c_2(E)=c_2(F)$ in $\CH_0(X)$.
    \item [(ii)] If $\srX=X$ is trivial, then any derived (anti-)autoequivalence preserves the class $\fro_X$.
    There exists a unique class $[E_0]\in \CH_0(M)$ such that
    $\bS_0\CH_0(M)=\ZZ[E_0]$ and $c_2(E_0)=k\fro_X$ for some integer $k$.
    \item [(iii)]If $\srX=X$ is trivial, then for any $[E],[F]\in M$ such that
    $c_2(E)+c_2(F)=2c_2(E_0)$, we have $$[E]_{i}=(-1)^i[F]_{i},$$ where  $[E]_{i}$ is the projection of $[E]\in\CH_{0}(M)$ to $\CH_{0}(M)_{i}$.
\end{enumerate}
\end{theorem}
\begin{proof}
The first statement  is the main theorem of \cite{MZ20}. 

For  (ii), the class $\fro_X$ is preserved by derived autoequivalences, as indicated by \cite[Corollary 2.6]{HMS09} and \cite[Corollary 7]{Voi15}, and $\DD$ acts as the identity on $\CH_0(X)$. The existence of $E_0$ is a consequence of \cite[\S 2.2]{SYZ20}, while its uniqueness is attributed to (i).

For (iii),  note that by Theorem \ref{bmotive}, we may assume that $M\cong X^{[n]}$. 
     Let $[x_1,x_2,\ldots,x_n]$ and $[y_1,y_2,\ldots,y_n]$ be two points in $X^{[n]}$.
Note that $$[x_1,x_2,\ldots,x_n]_{i}=\sum_{1\leq l_1\ldots\leq l_i\leq n} [x_{l_1}-\fro_X,\ldots, x_{l_i}-\fro_X,\fro_X,\ldots, \fro_X].$$
 We have to show that if $$\sum_{l=1}^n(x_l+y_l)=4n\fro_X,$$ then for any $1\leq i \leq n$, we have
    $$\sum_{1\leq l_1\ldots\leq l_i\leq n} [x_{l_1}-\fro_X,\ldots, x_{l_i}-\fro_X,\fro_X,\ldots, \fro_X]=(-1)^i\sum_{1\leq l_1\ldots\leq l_i\leq n} [y_{l_1}-\fro_X,\ldots, y_{l_i}-\fro_X,\fro_X,\ldots, \fro_X].$$ 
 Recall that X is covered by (singular) elliptic curves and contains an ample (nodal) rational curve. By employing the Riemann-Roch formula, we can find $y'\in X$ such that $$x+y'=2\fro_X,$$ for any $x\in X$ (\cf.~ \cite[Claim 0.2]{OG13} \cite[Lemma 2.6]{Voi15}).  Then for each $x_l$, there exists $y_l'\in X$ such that $x_l+y_l'=2\fro_X$ and by (1), we obtain
 $$[y_1,y_2,\ldots,y_n]\sim [y_1',y_2',\ldots,y_n'].$$ Therefore their projections to $\CH_{0}(X^{[n]})_{2i}$ are also equal. We obtain
 \begin{equation*}
     \begin{split}
        \sum_{1\leq l_1\ldots\leq l_i\leq n} [y_{l_1}-\fro_X,\ldots, y_{l_i}-\fro_X,\fro_X,\ldots, \fro_X]&=\sum_{1\leq l_1\ldots\leq l_i\leq n} [y_{l_1}'-\fro_X,\ldots, y'_{l_i}-\fro_X,\fro_X,\ldots, \fro_X]\\
         &= \sum_{1\leq l_1\ldots\leq l_i\leq n} [\fro_X-x_{l_1},\ldots, \fro_X-x_{l_i},\fro_X,\ldots, \fro_X]\\
         &=(-1)^i\sum_{1\leq l_1\ldots\leq l_i\leq n} [x_{l_1}-\fro_X,\ldots, x_{l_i}-\fro_X,\fro_X,\ldots, \fro_X],
     \end{split}
 \end{equation*}
and the assertion follows.
\end{proof}

\begin{remark}[Obstruction for twisted moduli space]
  It is an open problem whether a twisted derived equivalence preserves the BV class.  
  Recently, there are some progress in \cite{lzz24}.
\end{remark}

\subsection{Geometric examples}
Let us give some well-known examples of anti-symplectic involutions and give a geometric description of their actions on the Chow group. % We shall mention that different from the case of K3 surfaces, the quotient  $Y/\left<\iota\right>$ may have a quite large Chow group when $n>1$.  
\subsubsection{Beauville involutions}
Let $(X,H)$ be a polarized $K3$ surface of degree $H^2=2g-2$. Then we have $$X\hookrightarrow \PP^g.$$ Intersecting $X$ with $\PP^{g-2}\subset \PP^g$ gives an involution:
$$\iota:X^{[g-1]}\dashrightarrow X^{[g-1]}.$$ It acts on cohomology is the reflection around $H-E$ by \cite{o2005involutions}. Then $\iota$ induces a natural action $\iota_\ast:\CH_0(X)_{\hom}\to \CH_0(X)_{\hom}$ via the natural projection $\CH_0(X^{[g-1]})_{hom}\rightarrow \CH_0(X)_{\hom}.$

Since $\PP^{g-2}$ in $\PP^g$ are parameterized by the Grassmannian $\mathrm{Gr}(g-1,g+1),$  this immediately yields that
the map $\iota_\ast:\CH_0(X)_{\hom}\to \CH_0(X)_{\hom}$ is $-\mathrm{id}$. Then by Theorem \ref{thm:K3HK}(iii), we obtain
\begin{proposition}
    The induced action of $\iota$ on $\CH_0(X^{[g-1]})_{i}$ is $(-1)^i$.
\end{proposition}

We remark that this is also a special case of Theorem \ref{thm:K3HK}.

\subsubsection{O'Grady involution}
Let us consider hyper-K\"ahler varieties of $K3^{[n]}$-type  with an anti-symplectic involution  which are not Bridgeland moduli spaces on a twisted $K3$ surface. A typical example is a general polarized $K3^{[2]}$-type hyper-K\"ahler fourfold $(Y,H)$ of degree $2$, known as a double EPW sextic. It admits a morphism $$\phi_H: Y\rightarrow |H|^\vee=\PP^5,$$ which is a double cover ramified over a sextic hypersurface.
Additionally, there exists a natural anti-symplectic involution $\iota: Y\rightarrow Y$. Let $Z$ denote the fixed locus, which, according to \cite[Theorem 1.1]{Z23}, is a constant cycle surface.

Let $\mathfrak{o}_{Y}$ be the class of a point on $Z$. There is  an $(\pm 1)$-eigenspace decomposition
$$\CH_0(Y)\cong \ZZ \fro_Y \oplus\CH_0(Y)_{\hom}^-\oplus\CH_0( Y)_{\hom}^+.$$ In this case, Bloch's conjecture for $\iota$ is equivalent to the identity
\begin{equation}\label{eq:Bloch-EPW}
    \bS_1\CH_0(Y)=\ZZ \fro_Y \oplus\CH_0(Y)_{\hom}^-.
\end{equation}
This is currently unknown but $\ZZ \fro_Y \oplus\CH_0(Y)_{\hom}^-\subset \bS_1\CH_0(Y)$ is hidden in \cite[Proposition 5.4, Proposition 5.8]{Z23}.

\subsection{Characterization of Bridgeland moduli spaces}\label{subsection:c-Bridgeland}
For a $K3^{[n]}$-type hyper-K\"ahler variety $Y$,  let $(-,-)_Y$ be the Beauville-Bogomolov form on $\rH^2(Y,\ZZ)$. In \cite[Section 9]{Mark11}, Markman has found an extension $\Lambda(Y)$ of the lattice $\rH^2(Y,\ZZ)$ (see also \cite[Theorem 3]{add16}). We call $\Lambda(Y)$  the Markman-Mukai lattice of $Y$. 
\begin{theorem}\label{thm:HK-K3}The lattice
$\Lambda(Y)$ is isomorphic to the Mukai lattice of a $K3$ surface and the orthogonal complement $$\rH^2(Y,\ZZ)^\perp\subset \Lambda(Y)$$is generated by a primitive vector with square $2n-2$.
 If $Y=M_\sigma(\srX, v)$ is a moduli space of stable objects, there is a natural Hodge isometry between  $\rH^2(Y,\ZZ)$ and  $v^\perp\subset \widetilde\rH(\srX,\ZZ)$.
\end{theorem}

There is a natural weight two Hodge structure on $\Lambda(Y)$ and we let  $$\Lambda_{\rm alg}(Y):=\Lambda(Y)\cap \Lambda^{1,1}(Y).$$

We say $Y$ is a {\it Bridgeland moduli space} on a twisted K3 surface $\srX\to X$ if it is isomorphic to the coarse moduli space of stable objects  in $\rD^{(1)}(\srX)$. There is  a numerical characterization of Bridgeland moduli spaces for hyper-K\"ahler varieties of $K3^{[n]}$-type as below. 
\begin{theorem}[Addington-Huybrechts-Markman]\cite[Proposition 4]{add16}\cite[Lemma 2.6]{Huy17}\label{Thm:tmoduli}
    Let $Y$ be a $K3^{[n]}$-type hyper-K\"ahler variety with $n\geq 2$, then $Y$ is isomorphic to a coarse moduli space of stable objects on a (twisted) $K3$ surface if and only if $\Lambda_{\rm alg}(Y)$ contains a copy of $\rU$ (resp.~$\rU(k)$).
\end{theorem}
 Then we have
\begin{corollary}\label{cor:LF}
If either $Y$ has Picard number $g\geq4$ or admits a birational Lagrangian fibration, then $Y$ is a Bridgeland  moduli space. Moreover, if $\rank~\NS(Y)\geq 13$, it is a Bridgeland moduli space on some untwisted K3 surface $X$. 
\end{corollary}
\begin{proof}
For the first assertion,    the hypothesis on $Y$ ensures that   there exists an isotropic class   $L\in  \Lambda_{\alg}(Y)$. Then we can find  a class $H\in \Lambda_{\rm alg}(Y)$ such that $H\cdot L=m\neq 0$.  Then $\Lambda_{\rm alg}(Y)$  contains  the sublattice spanned by   $mH-\frac{(H^2)}{2}L$  and $L$ which is isomorphic to $\rU(m^2)$. The assertion follows directly from Theorem \ref{Thm:tmoduli}. 

For the second assertion,  if $\Sigma$ is an even lattice, we denote by $\ell(\Sigma)$ the minimal number of generators of its discriminant group  $\Sigma^\vee/\Sigma$.   Since $\Lambda_\alg(Y)$ is a primitive sublattice of the unimodular lattice $\Lambda(Y)$ of rank $24$,  we have $$\ell(\Lambda_\alg(Y))=\ell(\Lambda_{\alg}(Y)^\perp)\le 24-14=10.$$
It satisfies $\rank(\Lambda_{\alg}(Y))> 3+\ell(\Lambda_{\alg}(Y)) $. 
 According to  \cite[Corollary 1.13.5]{Ni79}, $\Lambda_\alg(Y)$ must contain  a hyperbolic plane $\rU$. By Theorem \ref{Thm:tmoduli},  this means  that $Y$ is isomorphic  to a moduli space of sheaves on a untwisted $K3$ surface.
\end{proof}

\subsection{Birational automorphism of finite order on HK of $K3^{[n]}$-type} 
Let $Y$ be a hyper-K\"ahler variety of $K3^{[n]}$-type. In \cite[Theorem 1.1]{dutta2022symplectic}, it has been showed that $Y$ is a Bridgeland moduli space if $Y$ admits a symplectic birational automorphisms of finite order.  Here we give a  substantial strengthening of this result, which classifies  birational automorphisms  of finite order on hyper-K\"ahler varieties of $K3^{[n]}$-type. 

We start by recalling a standard result on the action of birational automorphisms.
\begin{lemma}\cite[Lemma 4.10]{markman2008monodromy}\label{lemma:disc}
For $\phi\in \Bir(Y)$, the induced action $\phi^*:\rH^2(Y,\ZZ)\to \rH^2(Y,\ZZ)$  acts as $\pm\id$ on the discriminant lattice $A_Y=\rH^2(Y,\ZZ)^\vee/\rH^2(Y,\ZZ)$.  In particular, $\phi^\ast$ can be extended to a Hodge isometry in  $\rO(\Lambda(Y))$. 

\end{lemma}

Now we can define the invariant and coinvariant sublattices of $\Lambda(Y)$. 

\begin{definition}
Let $Y$ be a hyper-K\"ahler variety of $K3^{[n]}$-type and let $\phi\in \Bir(Y)$ be an element of finite order.  
\begin{itemize}
    \item let $\Lambda(Y)^{\phi}\subseteq \Lambda(Y)$  be the invariant lattice  under the action of $\phi^\ast:\Lambda(Y)\to \Lambda(Y)$. We define
    \begin{equation*}
        \Lambda(Y)_{\phi}=(\Lambda(Y)^{\phi})^\perp
    \end{equation*}
    to be the coinvariant lattice. 
\vspace{.1cm}
    \item  $\phi$ is  called {\it purely non-symplectic} if $\phi^k$ is non-symplectic for all positive integers $k<\mathrm{ord}(\phi)$.
\end{itemize}
\end{definition}

 The main result is 

\begin{theorem}\label{thm:finite-order}
With notations as above, then \begin{enumerate}
      \item $Y$ is isomorphic  to a moduli space of (twisted) sheaves on a $K3$ surface  except 
   \begin{itemize}
       \item  $1\leq \rank~\NS(Y)\leq 3$, $\phi$ is an anti-symplectic involution and  $\Lambda(Y)^{\phi}$ is positive definite. 
       \item  $\rank~\NS(Y)=1$ or $3$, $\phi$ is purely non-symplectic with ${\rm ord}(\phi)\in\{3,4,6\}$ and $\Lambda(Y)^{\phi}$ is positive definite. 
       \item  $\rank~\NS(Y)=1$ and  $\phi$ is purely non-symplectic with ${\rm ord}(\phi)\in\{23, 46\}$.
   \end{itemize}
(See Remark \ref{rmk:classification} for further classification of exceptional cases). 
\vspace{.1cm}

\item $Y$ is isomorphic to a moduli space of sheaves on a $K3$ surface with $\rank ~\NS(Y)\ge 13$ if  $\phi$ is (anti)-symplectic with  ${\rm ord}(\phi)\neq 2, 4$.

  \end{enumerate}

\end{theorem}
\begin{proof}  For (1), we  claim that $Y$ is a Bridgeland (twisted) moduli space  if  one of the following three conditions holds
   \begin{enumerate}
       \item  [i)] $\phi$ is symplectic;
       \item [ii)] $\mathrm{ord}(\phi)>2$ and $\rank~\NS(Y)\neq 1,3$;
       \item [iii)] $\mathrm{ord}(\phi)$ is divisible by one of the following numbers: $8$, $9$, $12$, an odd prime $p\neq 3,23$. 
   \end{enumerate}
By Corollary \ref{cor:LF}, it suffices to show $\Lambda_{\alg}(Y)$ contains an isotropic vector.  
   
In the first case, the assertion is
\cite[Theorem 1.1]{dutta2022symplectic}.

In the second case, by \cite[Lemma 4.1]{Og02},  we know that if $\rank(\Lambda_{\alg}(Y))=\rank(\NS(Y))+1$ is odd, then $\phi$ is either symplectic or anti-symplectic (See also \cite[Corollary 3.3.5]{Huy16}).  Then $\phi^2$ is a non-trivial symplectic birational automorphism. The assertion then follows from i).   When $\rank (\NS(Y))\geq 4$, there always exist isotropic classes in $\Lambda_\alg(Y)$.  

In the last case,  we may assume $\phi$ is purely non-symplectic with ${\rm ord}(\phi)=$ $8$, $9$, $12$, or an odd prime $p\neq 3, 23$ and $\rank (\NS(Y))=1$ or $3$.  Here we only treat the case ${\rm ord}(\phi)=p$ is a prime not equal to $3$ and $23$. The discussion for the rest of cases is similar. Set $$\rT(Y)=\Lambda_{\alg}(Y)^\perp\subseteq \Lambda(Y),$$ to be the transcendental lattice, then $\rank(\rT(Y))=23-\rank(\NS(Y)) \leq 22$. As $\phi$ is purely non-symplectic,   one must have  $p-1 \mid \rank(\rT(Y))=20$ or $22$. This implies   $p=5$ or $11$ and $\rank (\rT(Y))=20$. 
In this situation, $\rank (\Lambda_{\alg}(Y))=4$ and  $\Lambda_{\alg}(Y)^{\phi}=\Lambda_{\alg}(Y)$.  The invariant sublattice  $\Lambda(Y)^{\phi}\subseteq \Lambda_{\alg}(Y)$ is a $p$-elementary lattice by \cite[Theorem 6.1]{Kon92}.  By the classification of  $p$-elementary even lattices of signature $(2,2)$ (\cf.~ \cite[Chapter 15, Theorem 13]{CS99}), $\Lambda_{\alg}(Y)$ contains either $\rU$ or $\rU(p)$.  Note that our argument already implies that if  $ {\rm ord}(\phi)=23$ and $\phi$ is not symplectic,  then $\rank(\NS(Y))=1$.

To finish the proof, we only need to treat the case $\phi$ is purely non-symplectic with ${\rm ord}(\phi)\in\{2,3,4,6\}$ and $ \rank (\NS(Y)\cap \Lambda(Y)^{\phi})\geq 2$. By a similar argument in case iii), for cases ${\rm ord}(\phi)=2,4$ (resp. ${\rm ord}(\phi)=3, 6$), we are reduced to show that all $2$-elementary (resp. $3$-elementary) indefinite even lattices of rank at most $4$ are isotropic. This is a consequence of classification of $2$ and $3$-elementary lattices in \cite[Chapter 15, Theorem 13]{CS99}.
\vspace{.1cm}

For (2),  suppose $\phi$ is (anti)-symplectic of  ${\rm ord}(\phi)\neq 2,4$. Then $\phi^2$ is symplectic of order at least $3$ and the induced action on $A_Y=\rH^2(Y,\ZZ)^\vee/\rH^2(Y,\ZZ)$ is $\id$.  As shown in the proof of \cite[Theorem 1.1]{dutta2022symplectic},
the coinvariant lattice $\Lambda(Y)_{\phi^2}$
is a negative definite lattice contained in $\NS(Y)$ and  there is an isometry $f\in {\rm O}(N_{24})$ of the rank $24$ Leech lattice $N_{24}$ such that $\Lambda(Y)_{\phi^2}\cong (N_{24})_f:=(N_{24}^f)^\perp$ and ${\rm ord}(f)={\rm ord}(\phi^2)\ge 3$. Then by \cite[Table 1]{HM16}, we have $$\rank (\Lambda(Y)_{\phi^2})=\rank ((N_{24})_f)\ge 12,$$ which implies $\rank (\NS(Y))\ge\rank (\Lambda(Y)_{\phi^2})+1\ge 13$. Therefore, by Corollary \ref{cor:LF}, $Y$ is a Bridgeland moduli space on some untwisted K3 surface.

%Consider the action of $\varphi^\ast$ on  the unimodular lattice $\Lambda(Y)$, its invariant sublattice  and co-invariant sublattice are both $2$-elementary

%Consider the action of $\varphi^\ast$ on  the unimodular lattice $\Lambda(Y)$, its invariant sublattice  and co-invariant sublattice are both $p$-elementary (\cf~). 

%$\Lambda_{\alg}(Y)$ is isometric to one of the following lattices: $\rU\oplus\rU$, $\rU\oplus H_5$, $\rU\oplus \rU(5)$, $\rU(5)\oplus H_5$, $\rU(5)\oplus \rU(5)$, $\rU\oplus \rU(11)$, $\rU(11)\oplus \rU(11)$. Here $H_5={\left(\begin{array}{cc} 2 & 1 \\1 & -2 \end{array} \right)}.$ 
%By Corollary \ref{cor:LF}, it suffices to find an isotropic class in $\Lambda_{\rm alg}(Y)$ when $Y$ admits a symplectic birational self map of finite order. Let $\Lambda^{\varphi^*}\subset \Lambda$ be the invariant lattice and $$S_{\varphi^*}(\Lambda)=(\Lambda^{\varphi^*})^\perp\subset \Lambda.$$ Since $\varphi$ is symplectic, $\Lambda^{\varphi^*}\subset \Lambda_{\rm alg}.$ The proof of \cite[Theorem A]{dutta2022symplectic} has shown that the sublattice of $\Lambda_{\rm alg}$ generated by $S_{\varphi^*}(\Lambda)$ and $\rH^2(Y,\ZZ)^\perp$ always contains an isotropic class.

\end{proof}

\begin{remark}\label{rmk:classification}
For purely non-symplectic $\phi$ with ${\rm ord}(\phi)\in\{3,4,6,23,46\}$ and $\rank (\NS(Y))=1$, the Markman-Mukai lattice  $\Lambda_{\alg}(Y)$ is given as follows:
  \begin{enumerate}
       \item  [1)] ${\rm ord}(\phi)=3$, $\Lambda_{\alg}(Y)\cong {\left(\begin{array}{cc} 
2 & 1 \\
1 & 2 
\end{array} \right)}$;
       \item [2)] ${\rm ord}(\phi)=4$, $\Lambda_{\alg}(Y)\cong {\left(\begin{array}{cc} 
2 & 0 \\
0 & 2 
\end{array} \right)}$;

  \item  [3)] ${\rm ord}(\phi)= 6$, $\Lambda_{\alg}(Y)\cong {\left(\begin{array}{cc} 
2 & 1 \\
1 & 2 
\end{array} \right)}$, $\NS(Y)\cong\langle 2\rangle$ and $n=4$;

       \item [4)] ${\rm ord}(\phi)=23$, $\Lambda_{\alg}(Y)\cong {\left(\begin{array}{cc} 
2 & 1 \\
1 & 12 
\end{array} \right)}$, ${\left(\begin{array}{cc} 
4 & 1 \\
1 & 6 
\end{array} \right)}$;
\item [5)] ${\rm ord}(\phi)=46$, $\Lambda_{\alg}(Y)\cong {\left(\begin{array}{cc} 
2 & 1 \\
1 & 12 
\end{array} \right)}$, $\NS(Y)\cong\langle 2\rangle$ and $n=24$. 
   \end{enumerate}
 For cases 1), 2) and 4), the pair $(\NS(Y), n)$ have infinitely many possibilities satisfying certain numerical conditions.   We also refer the readers to  \cite{CC20} and \cite{CCC21} for some classification of the invariant sublattice of $\NS(Y)$ when  $n$ is small.       
\end{remark}

%% file: sec5.tex
\section{Proof of Theorem \ref{thm:blochHK} and generalizations}\label{sec:proof}

In this section, we use Theorem \ref{thm1} to deduce Theorem \ref{thm:blochHK}. The main idea is to show that (anti)-symplectic birational automorphisms on Bridgeland moduli spaces  are induced from autoequivalences. 
\subsection{Lift Hodge isometries to autoequivalences} 

In this section, we mainly deal with the case that $Y$ is a Bridgeland moduli space on some twisted $K3$ surface, where $\Lambda(Y)\cong \widetilde{\rH}(\srX,\ZZ)$.  
 We will need  the result below to lift Hodge isometries. 
\begin{lemma}\label{lemma:anti-lift}
    Let $\srX\to X$ and $\srX'\to X'$ be two twisted $K3$ surfaces.
    If there is an orientation-reversing Hodge isometry between 
    $$g:\widetilde\rH(\srX,\ZZ)\cong \widetilde\rH(\srX',\ZZ),$$
    Then there exists a derived anti-equivalence $$\Phi:\rD^{(1)}(\srX)\to \rD^{(1)}(\srX')^{\rm op},$$ such that $\Phi^{\widetilde\rH}=g$.
\end{lemma}
\begin{proof}
     Let $\srX''\to X$ be a twisted $K3$ surface with $[\srX'']=-[\srX']$. Then $$(-)^\vee:\widetilde\rH(\srX',\ZZ)\cong \widetilde\rH(\srX'',\ZZ)$$
     is an orientation-reversing Hodge isometry. Therefore $(-)^\vee\circ g$ can be lifted to a derived equivalence $\Psi:\rD^{(1)}(\srX)\to \rD^{(1)}(\srX'')$ according to \cite[Theorem 0.1]{Huy06}. Then $\Phi=\DD\circ \Psi$ is the desired anti-equivalence, where $\DD:\rD^{(1)}(\srX'')\to \rD^{(-1)}(\srX'')^{\rm op}\cong \rD^{(1)}(\srX')^{\rm op}$ is the derived dual functor we used  before.
\end{proof}
 
 \subsection{Lifting automorphisms}
 
 Let $\srX\to X$ be a  twisted $K3$ surface and $v\in \widetilde{\rH}(\srX,\ZZ)$ a primitive vector. Assume that $v^2\geq 0$ and $\sigma\in \mathrm{Stab}^\dagger(\srX)$ is $v$-generic. Let $M=M_\sigma(\srX, v)$ be the coarse moduli space of $\sigma$-stable objects in $\rD^{(1)}(\srX)$.  Then we obtain the following, which is essentially a consequence of \cite{bayer2014mmp}.
\begin{proposition}\label{mortoder}Assume $\srX=X$ is untwisted. 
For any $\phi\in \Aut(M)$, there exists an (anti-)autoequivalence
$\Phi:\rD^b(X)\rightarrow \rD^b(X)$  (resp.~$\rD^b(X)\rightarrow\rD^b(X)^{\rm op}$) and a $v$-generic stability condition $\tau$ such that  
$\Phi$ induces $\phi$.
\end{proposition}
\begin{proof}
     When $v^2=0$, the statement is trivial. We assume that $v^2\geq 2$ and identify $\rH^2(M,\ZZ)$ with $v^\perp\subset \widetilde\rH(X,\ZZ)$. The action $\phi^*$ on $\rH^2(M,\ZZ)$ can be lifted to an isometry on $\widetilde\rH(X,\ZZ)$ by Lemma \ref{lemma:disc}. It is orientation-preserving if the action on $A_M$ is the identity and orientation-reversing if the action on $A_M$ is $-\id$.
     
     Let $h\in {\rm Amp}(M)$ be an ample class.
     According to the description of ${\rm Amp}(M)$ in \cite[Theorem 12.1]{bayer2014mmp}, $\left<h.\alpha\right>\neq 0$ for any $\alpha\in \NS(M)$ with $\alpha^2\geq -2$.
     Then the ampleness implies that 
     $$\Omega_Z:=v-ih\in \cP_0^+(X).$$
     Assume that there exists $w\in \widetilde\rH(X,\ZZ)$ such that $\frac{\left<\Omega_Z,v\right>}{\left<\Omega_Z,w\right>}\in \RR$.  Then either $w^2\leq -4$ or $w=kv$. This implies that any stability condition in $\cZ^{-1}(\Omega_Z)$ is $v$-generic.

Now we select $h_0\in{\rm Amp}(M)$ and let $h_1:=\phi^*(h_0)$, which remains an ample class. Define
$$\Omega_{Z_t}:=v-i((1-t)h_0+th_1)~{\rm for}~t\in \left[0,1\right].$$ Since $(1-t)h_0+th_1\in {\rm Amp}(M)$, the stability conditions in $\cZ^{-1}(\Omega_{Z_t})$ are all $v$-generic. Hence, the line segment $\left[\Omega_{Z_0},\Omega_{Z_1}\right]\subset \cP_0^+(X)$ can be lifted to a path $\gamma:\left[0,1\right]\to {\rm Stab}^\dagger(X)$ that lies entirely within a chamber.

Set $\gamma(0)=\tau.$ 
 Recall that there is a natural map 
     $$\begin{tikzcd}
{\rm Stab}^\dagger(X) \arrow[r, "\ell"] & \NS(M)
\end{tikzcd},$$
as defined in \cite{bayer2014mmp}, such that for any generic $\sigma'\in {\rm Stab}^\dagger(X)$, $\ell(\sigma')$ is given by the composition:
$${\rm Stab}^\dagger(X)\xrightarrow{\cZ}\widetilde\rH(X,\ZZ)_\CC\xrightarrow{I}v^\perp\xrightarrow{\theta_{\sigma,v}}{\rm Pos}(M)\xrightarrow{W}{\rm Mov(M)},$$ where $I(\Omega_{Z})={\rm Im~}\frac{-\Omega_{Z}}{(\Omega_{Z},v)}$ and $W$ is the identity on the ample cone ${\rm Amp}(M)$ (\cf.~\cite[Theorem 10.2]{bayer2014mmp}). According to \cite[Theorem 1.2]{bayer2014mmp}, for any $v$-generic stability condition $\sigma'\in {\rm Stab}^\dagger(X)$, the coarse moduli space $M_{\sigma'}(\srX, v)$ is the birational model of $M$ corresponding to $\ell(\sigma').$ It can be verified that $\ell(\tau)$ is proportional to $h$ and there is a canonical identification $M_\tau(\srX, v)\cong M.$

Finally, since $\phi^*(\Omega_{Z_0})=\Omega_{Z_1}$, we can choose a derived equivalence $\Phi$ so that $\Phi^{\widetilde\rH}=\phi^*$ and $\Phi(\tau)=\gamma(1)$ by Theorem \ref{Thm:lift-der}, Lemma \ref{lemma:anti-lift} and Theorem \ref{deck}. The isomorphism $\varphi$ is induced by $\Phi$, hence $\varphi^*=\Phi^{\widetilde\rH}=\phi^*$.

\end{proof}
\begin{remark}
From the proof, one can see that  if the second statement of Theorem \ref{deck} holds for twisted $K3$ surfaces,  Proposition \ref{mortoder} holds for twisted $K3$ surfaces  as well. 
\end{remark}

\subsection{Lifting birational automorphisms} 

For birational automorphisms, we work on general Bridgeland moduli spaces and  obtain a slightly weaker result. 
\begin{proposition}\label{prop:bir-de}
 Let $M=M_\sigma(v)$ be a Bridgeland moduli space on a twisted K3 surface $\srX\to X$.    For any $\phi \in \mathrm{Bir}(M)$, there exists a derived (anti-)automorphism $$\Phi:\rD^{(1)}(\srX)\rightarrow \rD^{(1)}(\srX) ~(\hbox
{resp. ~} \rD^{(1)}(\srX)\rightarrow\rD^{(-1)}(\srX)^{\rm op}),$$ 
    such that $\phi^{\widetilde\rH}=\Phi^{\widetilde\rH}$ (resp. $\phi^{\widetilde\rH}=-\Phi^{\widetilde\rH}$) and there is an open subset $U\subset M$ and $\phi([E])=[\Phi(E)]$ for any $[E]\in U$.
\end{proposition}
\begin{proof}
As shown before, we can lift the action $\phi^*:\rH^2(M,\ZZ)\to \rH^2(M,\ZZ)$ to an autoequivalence $\Psi:\rD^{(1)}(\srX)\rightarrow \rD^{(1)}(\srX)$ if the action on $A_Y$ is $\id$.   If its action  on $A_Y$ is $-\id$,  we can lift $\phi^\ast$ to an anti-equivalence $\Psi:\rD^{(1)}(\srX)\rightarrow \rD^{(-1)}(\srX)^{\rm op}$ by Lemma \ref{lemma:anti-lift}. 
    
    The equivalence $\Psi$ (resp. $\Psi[1]$) induces an isomorphism $M\rightarrow M_{\Psi(\sigma)}(v)$. Due to the wall crossing  result \cite[Theorem 1.1]{bayer2014mmp}, there exists a derived (anti-)autoequivalence $\Psi'$ and a birational map $$ M_{\Psi(\sigma)}(v)\dashrightarrow M,~E\rightarrow \Psi'(E)$$ defined over an open subset $U\subset M_{\Psi(\sigma)}(v).$
    Composing it  with $M\cong M_{\Psi(\sigma)}(v)$, we obtain another birational map
    $$\varphi:M\cong M_{\Psi(\sigma)}(v)\dashrightarrow M,$$ which sends $[E]$ to $[\Phi(E)]$ for $E$ lying in some open subset of $M$, where $\Phi=\Psi'\circ \Psi$ (resp. $\Phi=\Psi'\circ \Psi[1]$).
    By \cite[Lemma 10.1]{bayer2014mmp}, the actions of $\phi^*$ and $\varphi^*$ on $\rH^2(M,\ZZ)$ differ by products of prime exceptional reflections. Then according to \cite[Theorem 6.18(5)]{markman2008monodromy}, $\phi^*$ and $\varphi^*$ must coincide. Hence $\varphi=\phi$ and $\phi^{\widetilde\rH}=\Phi^{\widetilde\rH}$ (resp. $\phi^{\widetilde\rH}=-\Phi^{\widetilde\rH}$).
 \end{proof}

\subsection{Bloch's conjecture for birational automorphisms} 
\begin{theorem}\label{thm:bloch-HK-strong}
If $Y$ is of $K3^{[n]}$-type,  then Conjecture \ref{conj3} holds  if one of the following conditions holds
\begin{enumerate}
\item [(1)]  $Y$ is a moduli space of stable objects on some $K3$ surface $X$ and $\rank(\NS(Y))\geq 4$.
\item [(2)]  $\phi$  is symplectic and $\phi^\ast:\Lambda_{\alg}(Y)\to \Lambda_{\alg}(Y)$ maps some isotropic vector $v$ to $\pm v$.  
\end{enumerate}
\end{theorem}
\begin{proof}
    According to Corollary \ref{cor:LF}, in both cases, $Y$ is a Bridgeland moduli space on some twisted K3 surface  $\srX\to X$ and  we have an identification $\Lambda(Y)\cong \widetilde{\rH}(\srX)$.   By Proposition \ref{prop:bir-de}, $\phi$ is induced by an (anti-)autoequivalence $\Phi:\rD^{(1)}(\srX)\to \rD^{(1)}(\srX)$ (resp. $\rD^{(1)}(\srX)\rightarrow\rD^{(-1)}(\srX)^{\rm op}$) and the 
    birational map $\phi$ is induced by $\Phi.$ 
   % Let $E$ on $\srX$ be a general point on $M$ and it suffice to prove the result for general points in $M$ by the moving lemma.

    In both cases, if $\phi$ is symplectic, we have $\Phi^\CH=\id$
    by Theorem \ref{thm:thm2} and Theorem \ref{thm:lagk3}. Moreover, there is a fixed point $[E_0]\in Y$.
    This implies that $$c_2(E)-c_2(E_0)=c_2(\Phi(E))-c_2(\Phi(E_0))$$ for general $[E]\in Y$. 
    By Theorem \ref{thm:K3HK}(i), we have $$ \phi([E])=[\Phi(E)]=[E]$$
    in $\CH_0(Y)$. This proves $\phi$ acts as identity on $\CH_0(Y)$. 

   For (1), the remaining case is when $\phi$ is anti-symplectic and $\srX=X$ is untwisted. By Theorem \ref{thm:thm2}, we have $\Phi^\CH=-\id$. Take a point $[E_0]\in Y$ such that $c_2(E_0)=k\fro_X$ as in Theorem \ref{thm:K3HK}, then $\Phi^\CH=-\id$ implies that $$c_2(E)+c_2(\Phi(E))=2c_2(E_0).$$ By Theorem \ref{thm:K3HK}(iii), we get 
   \begin{equation*}
       \phi([E])_i=(-1)^i[E]_i
   \end{equation*}
which means that $\phi$ acts as $(-1)^i \id $ on $\CH_0(Y)_i$.

\end{proof}

%As above, we assume $\phi$ is induced by an (anti-)autoequivalence $\Phi$.    According to Theorem \ref{thm:K3HK} (i)(iii), it suffices to show $\Phi^{\CH}=\pm\id$.  By our assumption, there exists an isotropic vector $w\in  v^\perp\subseteq \widetilde{\NS}(\srX)$ which is fixed by  $\Phi^{\widetilde{\rH}}$. It follows from  Theorem \ref{lagk3} that  $\Phi^{\CH}=\id$ if $\phi$ is symplectic and $\Phi^{\CH}=-\id$ if $\phi$ is anti-symplectic.

%\begin{theorem}\label{thm:bloch-LF}If $\phi\in \Bir(M_\sigma (\srX,v))$ preserves a birational Lagrangian fibration, then Conjecture \ref{conj3} for $\phi$  holds if either $\phi$ is symplectic or $\srX\to X$ is untwisted. \end{theorem}\begin{proof}\end{proof}

%\begin{corollary}\label{cor:blochall} Let $Y$ be a $K3^{[n]}$-type hyper-K\"ahler variety. If the Markman-Mukai lattice $\Lambda_{\alg}(Y)$ contains a hyperbolic lattice and $\rank (\NS(Y) )\neq 2$ and $3$, then Conjecture \ref{conj3} holds. \end{corollary}\begin{proof}   This follows directly from Theorem \ref{thm2} and Theorem \ref{thm:HK-K3}. \end{proof}

\subsection{Proof of Theorem \ref{thm:bloch-finite}}

If $\mathrm{ord}(\phi)\neq 2,4$, this follows from Theorem \ref{thm:finite-order} (2) and Theorem \ref{thm:bloch-HK-strong} (1).

If $\phi$ is a symplectic involution,  as the $2$-elementary  coinvariant lattice $\Lambda(Y)_{\phi}\subseteq \Lambda_{\alg}(Y)$ is not isomorphic to $E_8(-2)$,   there are only two  possibilities:
\begin{itemize}
 \item $\Lambda(Y)_{\phi}$ is indefinite and the action of $\phi$ on $A_Y$ is $-\id$;   
 \item $\Lambda(Y)_{\phi}$ is negative definite and  $\rank ~\Lambda(Y)_{\phi}\geq 12$ .
\end{itemize}
(cf.~\cite[Lemma 2.4, 2.9]{dutta2022symplectic}).
In the first situation,  $\Lambda(Y)_\phi$ has an isotropic vector.  In the second situation, $Y$ is a Bridgeland moduli space on a untwisted $K3$ surface with Picard number $\geq 13$. The assertion follows from Theorem \ref{thm:bloch-HK-strong}.

\subsection{Proof of Theorem \ref{thm:lag-cons}}  It suffices to prove the loci $\mathrm{Fix}(\iota)$ is a constant cycle subvariety. 
For $n=1$, this is well-known. For $n=2$ and $\rank(\NS(Y)^{\iota})=1$, according to \cite[Theorem 3.11]{CCC21},  $Y$ is a double EPW sextic and the assertion follows from \cite[Theorem 1.1]{Z23}.

In all the remaining cases,  by Theorem \ref{thm:finite-order} (1), we know that $Y$ is a Bridgeland moduli space on some twisted $K3$ surface $\srX\to X$. 
According to  Proposition \ref{mortoder}, we can assume that there exists an (anti-)autoequivalence $\Phi:\rD^{(1)}(\srX)\rightarrow \rD^{(1)}(\srX)$ such that $\iota([E])=[\Phi(E)]$. 
  Since $\rank(\NS(Y)^{\iota})\geq 2$, the invariant lattice of $\Phi^{\widetilde{\NS}}$ on $\widetilde{\NS}(\srX)$ is isomorphic to $ \Lambda_\alg(Y)^{\iota}$ which is an indefinite $2$-elementary lattice of rank $>1$.   The lattice $\Lambda_\alg(Y)^{\iota}$ must contain an isotropic vector.  Thus we can get
$$\Phi^\CH=-\id,$$
 by applying Theorem \ref{thm:lagk3}.

     For any $[F_1]$ and $[F_2]$  lying in the fixed locus of $\iota$, we have  $$c_2(F_1)-c_2(F_2)=-(c_2(\Phi(F_1))-c_2(\Phi(F_2))=-(c_2(F_1)-c_2(F_2)).$$ This gives $c_2(F_1)=c_2(F_2)$ as $\CH_0(Y)$ is torsion free. Then the constantness follows from Theorem \ref{thm:K3HK}(i).
   % The dimension of the fixed loci is computed in \cite[Lemma 1]{Bea11} and \cite[Corollary 1.2]{Voi16} has shown that constantness implying Lagrangian.
    \qed

%% file: sec6.tex
\section{Integral Cartan-Dieudonne}\label{sec:kneser}
 
 In this section, we generalize Kneser's result in \cite{Kn81} and show that the stable special orthogonal group of an extended Mukai lattice is generated by reflective involutions. This leads to Theorem \ref{thm:thm2}.
 \subsection{Notation and conventions}
  Let $A=\ZZ$ or $\ZZ_p$ and let $B$ be the fraction field of $A$.  Denote by $A^\times$  (resp.~$B^\times$) the group of  invertible elements in $A$ (resp.~$B$). 
 Let $\Lambda$ be a lattice over $A$ with a bilinear form $(-,-)$ and set $q(x)=\frac{1}{2}(x,x)$. Then we set 
 \begin{itemize}
     \item $\rO(\Lambda)$:  the orthogonal group of $\Lambda$,
     \item  $\rS\rO(\Lambda)$:  the special orthogonal group,
     \item $\widetilde{\rO}(\Lambda)$: the stable orthogonal group which consists of elements in $\rO(\Lambda)$ acting trivially on the discriminant group $\Lambda^\vee/\Lambda$. 
 \end{itemize}
 If $\Lambda$ is integral over $\ZZ$,  we say $\Lambda$ satisfies condition $(\ast)$ if 
\begin{enumerate}
    \item [i)]  there exists vectors $\upsilon,\omega\in \Lambda$ such that  \begin{center}
         $q(\upsilon)=1$ and $q(\omega)=-1$. 
     \end{center}  
    \item [ii)] the Witt index of $\Lambda_\RR$ is $\geq 2$ and the Witt index of $\Lambda_\QQ$ is $\geq 1$.
\end{enumerate}

For $n\in A$, we use $\Lambda(n)$ to denote the twist lattice defined by the bilinear form $n(-,-)$ on $\Lambda\times \Lambda$.  Let $\rU$ be the hyperbolic plane whose gram matrix of the bilinear form  under a standard basis is given by
\begin{equation}\label{eq:hyp-basis}
\left(\begin{array}{cc}
0&1\\
1&0
\end{array}\right).
\end{equation}

\subsection{Orthogonal group and its subgroup} With the notation as above, 
 for each vector $b\in \Lambda$ with $q(b)\in A^\times$,  let  $$s_b(x)=x-(x,b)q(b)^{-1} b\in \rO( \Lambda)$$ be the reflection around  $b$ and we can further define
 \begin{itemize}
      \item $\rR(\Lambda)$: the subgroup generated by reflections $s_v$ with $q(v)\in A^\times$,
      \item $\rR^{\pm}(\Lambda)\subseteq \rR(\Lambda)\cap \rS\rO(\Lambda)$: the subgroup generated by  $s_a\circ s_b$ with $q(a)=q(b)=\pm 1$,
      \item $\rR^{+}(\Lambda)\subseteq \rR^{\pm}(\Lambda)$: the subgroup generated by the product $s_a\circ s_b$ with $q(a)=q(b)=1$. 
  \end{itemize}

  Next, recall that there is a spinor norm map   $$spin_{\Lambda}: \rO(\Lambda)\to B^\times /(B^\times )^2$$
   defined by sending a reflection $s_b$ to $q(b)$. We let $\rO^\dagger(\Lambda)$ be the subgroup consisting of elements with spinor norm $1$ and $\rS\rO^\dagger(\Lambda)=\rO^\dagger(\Lambda)\cap \rS\rO(\Lambda)$.
   
 Suppose $\Lambda$ is an even lattice defined over $\ZZ$.   Let us mention the relation between $\rO^\dagger(\Lambda)$ and the subgroup preserving the orientations that occurred in the introduction.    Define $\rO^+(\Lambda)$ to be the subgroup consisting of elements preserving the orientation. In general, $\rO^+(\Lambda)$ is not necessarily the same as  $\rO^\dagger(\Lambda(-1))$.  If  we consider the subgroup $$\widetilde{\rO}^+(\Lambda):=\rO^+(\Lambda)\cap \widetilde{\rO}(\Lambda),$$  then there is an identification given as below.

\begin{lemma}
$\widetilde{\rO}^+(\Lambda)=\widetilde{\rO}^\dagger(\Lambda(-1))$.
\end{lemma}
\begin{proof}
As the reflection $s_b$ reverses orientation if $q(b)>0$ and preserves orientation if $q(b)<0$,  we have  $\widetilde{\rO}^\dagger(\Lambda(-1))\subseteq\widetilde{\rO}^+(\Lambda)$. On the other side, for $g\in \widetilde{\rO}^+(\Lambda)$.
it is a standard result of Nikulin that $\Lambda$ can be primitively embedded into an even unimodular lattice $\Lambda'$ such that $g$ can be extended to an orthogonal transform $g'\in \rO(\Lambda')$ and $g'|_{\Lambda^\perp}=\id$ (\cf.~\cite{Ni79}). Under this construction,  $g'$ is lying $ \widetilde{\rO}^+(\Lambda')=\rO^+(\Lambda')$ as it preserves the orientation of $\Lambda'$.

Viewing $g'$ as an element in $\rO(\Lambda'(-1))$, then    the spinor norm  $spin_{\Lambda'(-1)}(g')$ is  $\pm 1 ~{\rm mod}~ \QQ^*,$
by \cite[Theorem 6.2]{GMC02}. As  $g$ preserves orientation of $\Lambda'$, one must have  $spin_{\Lambda'(-1)}(g')=1$ and  hence $g'\in \rO^\dagger(\Lambda'(-1))$. Since  $spin_{\Lambda(-1)}(g)=spin_{\Lambda'(-1)}(g')$, it follows that $g\in \widetilde{\rO}^\dagger(\Lambda(-1))$. 
\end{proof}

\subsection{Generalization of Kneser's results}
We first extend the main result of \cite{Kn81} by allowing the reflection around  vectors with norm $-1$.  

\begin{theorem}\label{thm:Kneser2}
Let $\Lambda$ be an even lattice with the quadratic form $q$ and set $\Lambda_p=\Lambda\otimes \ZZ_p$ satisfying the condition $(\ast)$. 
\begin{enumerate}
    \item  If $\rank ~\Lambda \geq 4$ and $\rR^\pm(\Lambda)$ is a congruence subgroup,   we have \begin{equation}
    \rR^\pm (\Lambda)=\bigcap_p (\rS\rO(\Lambda)\cap \rR^{\pm}(\Lambda_{p})).
\end{equation}
In particular, this holds when $\rank~\Lambda\geq 5$. 
    \item $\rR(\Lambda_p)=\widetilde{\rO}(\Lambda_p)$ if either $p\neq 2$ or $p=2$  and $\bar{q}$ is not equivalent to  $x_1x_2$ and $x_1x_2+x_3x_4$.
    \item   $\rR(\Lambda_p)\cap \rS\rO^\dagger(\Lambda_p)=\rR^\pm (\Lambda_p)$  if one of the following conditions holds
    \begin{enumerate}
        \item [(3.a)]$p\neq 2$,
        \item [(3.b)] $p=2$,  $q\mod 2$ is not equivalent to  $x_1^2$.
    \end{enumerate} 
\end{enumerate}
\end{theorem}
\begin{proof}
For (1),   the proof is essentially the same as  \cite[Satz 1]{Kn81}. Let $v,w\in \Lambda $ with  \begin{center}
    $q(v)=1$ and $q(w)=-1$.
\end{center}  By our assumption, it contains a principal congruence subgroup of level $m$. Let $S$ be the collection of prime factors of $m$.  For each $p\in S$, we can write $$g=\prod\limits_{i=1}^k (s_{c^{(p)}_{2i-1}}\circ s_{c_{2i}^{(p)}})\in\rS\rO(\Lambda_p), $$
 for some  $c_i^{(p)}\in \Lambda_p$ with $q(c_{2i-1}^{(p)})=q(c_{2i}^{(p)})=\pm 1$. Here,  we can make $k$ and the sign of $q(c_i^{(p)})$ to be  independent of  $p\in S$ by suitably adding the trivial identity $s_v\circ s_v=\id$ and $s_w \circ s_w=\id$.  For $p\notin S$, we let $c_i^{(p)}$ be either $v$ or $w$ (depending on the norm of $c_i^{(p)}$). By \cite[Theorem 104:3]{OM00}, one can find $c_i\in \Lambda $ to approximate $\prod\limits_p c_i^{(p)}$ at places in $S$. Let $h=\prod\limits_{i=1}^k (s_{c_{2i-1}}\circ s_{c_{2i}})\in \rR^\pm(\Lambda)$. Then we have 
$h^{-1}\circ g $ is contained in the principal congruence subgroup of level $m$ and it follows that $g\in \rR^\pm(\Lambda)$.

For (2), this is \cite[Satz 2]{Kn81}.

For (3),  when $p\neq 2$ or $p=2$ and $q \mod 2$ is not equivalent to $x_1^2$,  Kneser has shown in \cite[Lemma 1]{Kn81} that 
\begin{equation}
    \rR(\Lambda_p)\cap \rS\rO^\dagger(\Lambda_p)=\left<s_a\circ s_c |~q(a)=q(c)\in \ZZ_p^\times\right>. 
\end{equation}
When $p>3$, this was proved in \cite[Lemma 2]{Kn81} that $$ \rR(\Lambda_p)\cap \rS\rO^\dagger(\Lambda_p)=\rR^+(\Lambda_p)=\rR^{\pm}(\Lambda_p).$$ When  $p=3$,  as $\ZZ_p^\times /(\ZZ_p^\times)^2=\{\pm 1\}$,  one can find 
$e\in \ZZ_p^\times $ such that $e^2=\pm q(a)$.  Note that $s_a\circ s_c=s_{a/e}\circ s_{c/e}$ and $q(a/e)=q(c/e)=\pm 1$. 
It follows that $\rR^\pm(\Lambda_p)=\rR(\Lambda_p)\cap \rS\rO^\dagger(\Lambda_p)$.

When $p=2$,  as $\ZZ_2^\times /(\ZZ_2^\times)^2=\{\pm 1, \pm 3\} \cong (\ZZ/2\ZZ)^2$, it suffices to show that for  any $a,c\in \Lambda_p$ with $q(a)=q(c)=\pm 3$, we have  $$s_a\circ s_c=\prod s_{a_{2i-1}}\circ s_{a_{2i}},$$ for some $a_i\in\Lambda_p$ satisfying $q(a_{2i-1})=q(a_{2i})=\pm 1$. If $q\mod 2$ is not equivalent to $x_1^2, x_1x_2, x_1^2+x_2x_3$,  $x_1x_2+x_3x_4$ or $x_1^2+x_2x_3+x_4x_5$, Kneser has actually proved in \cite[Lemma 1\& Lemma 2]{Kn81} that $s_a\circ s_c$ is a product of $s_{a_{2i-1}}\circ s_{a_{2i}}$ with $q(a_{2i-1})=q(a_{2i})=1$. For (3.b),  we are thus reduced to consider the case when $q\mod 2$ is equivalent to one of the following forms
\begin{equation}\label{eq:3form}
    x_1x_2,  x_1^2+x_2x_3, ~ x_1x_2+x_3x_4,  ~x_1^2+x_2x_3+x_4x_5.
\end{equation} Let $\Omega$ be the  the sublattice of $\Lambda_p$ spanned by $a$ and $c$.  Then we have 

\begin{enumerate}
    \item [i)] if  $\Omega$ is unimodular, i.e. $(a,c)\equiv 1 \mod 2$, then there exists $f\in \Omega$ such that $q(f)=1$.  As $\rank~\Omega=2$, we have $$s_f\circ(s_a\circ s_c)=s_g,$$ for some $g\in \Omega$. Since the spinor norm of $s_g$ is $1$, we have $q(g)=1$ as well. 

    \item [ii)]  if $(a,c)\equiv 0\mod 4$,   there does not exist $f\in \Omega$ with $q(f)=1$, but there exists $f'\in \Omega$ with $q(f')=-1$. We have $s_{f'}\circ(s_a\circ s_c)=s_{g'}$ for some $g'\in \Omega$ and $q(g')=-1$.  

        \item [ii')]  Similarly, if $q(a)=-q(b)=3$ and  $(a,b)\neq 2\mod 4$ for some $b\in \Lambda_p$,   there exists a vector $f$ lying in the sublattice $\left<a,b\right>$ spanned by $a$ and $b$ such that $q(f)=1$. Then $s_f\circ(s_a\circ s_b)=s_g$ for some $g\in \left<a,b\right>$ with $q(g)=-1$.

    \item [iii)] if $(a,c)\equiv 2 \mod 4$,  by i), ii) and ii'), we only need to find vectors $b_0,\ldots, b_k \in \Lambda_p$ such that $b_0=a, b_k=c$ and 
    \begin{equation}\label{eq:chain}
        q(b_i)=\pm 3 ~\hbox{and}~(b_i,b_{i+1})\not\equiv 2 \mod 4.
    \end{equation}
Set $\overline{\Lambda}=\Lambda_2/8\Lambda_2$. Then this is equivalent to find  vectors $\bar{b}_i\in \overline{\Lambda}$ with $\bar{b}_0=\bar a$ and $\bar b_k=\bar c$ such that   $q(\bar{b}_i)\equiv \pm 3\mod 8$ and $(\bar{b}_i,\bar{b}_{i+1})\not\equiv \pm 2 \mod 8$.  This can be achieved by the following computation.

1)  If $q\mod 2$ is $x_1x_2$, we can write  $\overline{\Lambda}$ as   $\overline{\rU}\oplus \overline{\Omega}(2)$ (\cf.~\cite[\S 15.4]{CS99}). Let $\{\bar{e},\bar{f}\}$ be the standard basis of $\overline{\rU}$ and set $$a\equiv  \alpha_1 \bar e+\beta_1\bar f +L, ~c\equiv  \alpha_2 \bar e+\beta_2 \bar f +L',$$ for $L, L'\in \overline \Omega(2)$.  Then $\alpha_1\beta_2+\alpha_2\beta_1 \notin (\ZZ/8\ZZ)^\times $ and $\alpha_i\beta_i \in (\ZZ/8\ZZ)^\times$.  Hence one can take $$\bar b_i=\begin{cases}\alpha_i \bar e- \beta_i \bar f~\in \overline \rU,~ &\hbox{if}~\alpha_i\beta_i=\pm 3\mod 8;\\
\alpha_i \bar e+3\beta_i \bar f\in \overline \rU,  &\hbox{if}~\alpha_i\beta_i=\pm 1\mod 8.
\end{cases}$$
This proves the assertion. 

2) If $q\mod 2$ is $x_1x_2+x_3x_4$,  as above, 
we can assume  $\overline \Lambda=\overline{\rU}^{\oplus 2}\oplus \overline{\Omega}(2)$.  Similar computation shows that we can find $$\bar b_1, \bar b_2 \in \overline{\rU}\oplus \overline{\rU},$$ such that  $q(\bar b_i)= 3\mod 8$, $(\bar a, \bar b_1)\neq \pm 2\mod 8$ and $(\bar c, \bar b_2)\neq  \pm 2\mod 8$. The assertion then follows from the fact   $s_{\bar b_1}\circ s_{\bar b_2}\in \rR^{\pm}(\overline{\rU}\oplus \overline{\rU})$ (\cf.~\cite{Wall63}).

3)   If $q\mod 2$ is $x_1^2+x_2x_3$ or $x_1^2+x_2x_3+x_4x_5$, one can conclude the assertion by similar computations. There is also another way to see this.  
Note that  $\overline{\Lambda}$ can be written  as the direct sum $\overline{\Lambda}=\overline{\rI}\oplus \overline{\Lambda}'$ where $m\in (\ZZ/8\ZZ)^\times $  and  $\overline{\Lambda}'$ satisfies the hypothesis of   1) or 2) respectively.  Then one can also prove the assertion by reducing it to  1) or 2) via the use of  the Eichler  transvection introduced in Example \ref{exm:prop-invo}. See Lemma \ref{lem:rel-involution} and the proof of Theorem \ref{thm:kne3} (3').  
\end{enumerate}
\end{proof}

\begin{corollary}\label{cor:ref}
Let $\Lambda$ be an even lattice satisfying $(\ast)$ and $\rank (\Lambda)\geq 5$.  Then $\rR^\pm(\Lambda)=\widetilde{\rS\rO}^\dagger(\Lambda)$ if $q\mod 2$ is not equivalent to $x_1^2, x_1x_2$  and $x_1x_2+x_3x_4$. 
\end{corollary}
\begin{proof}
One may regard $\rR^\pm(\Lambda)$ as a nontrivial normal subgroup of the Spin group  $\mathrm{Spin}(\Lambda)$.      Under the hypothesis, it has been proved  in  \cite{Kn79,Va73} that $\rR^\pm(\Lambda)$ is  a congruence subgroup of $\mathrm{Spin}(\Lambda)$. The assertion follows directly from Theorem \ref{thm:Kneser2} and the fact 
      \begin{equation}
          \widetilde{\rS\rO}^\dagger(\Lambda)= \bigcap\limits_p (\rO(\Lambda) \cap  \widetilde{\rS\rO}^\dagger(\Lambda_p)).
      \end{equation}
     See the proof in \cite[Satz 4]{Kn72}. 
\end{proof}

\begin{example}[Symplectic involutions]
Combined with Theorem \ref{thm1}, the following computation gives an explicit proof of Bloch's conjecture for symplectic involutions.  Let $\Lambda=\ZZ\omega\oplus \rU(-1)\oplus E_8(-2)$ be the even lattice  whose  intersection matrix is given by   
 $${\left(\begin{array}{ccccccccccc} 
 2n & 0 & 0 & 0 & 0 & 0 & 0 & 0 & 0 & 0 & 0\\
0 & 0 & -1 & 0 & 0 & 0 & 0 & 0 & 0 & 0 & 0\\
0 &-1 & 0 & 0 & 0 & 0 & 0 & 0 & 0 & 0 & 0\\
0 &0& 0 & - 4 & 2 & 2 & 0 & 2 & 0 & 0 & 0\\
0 &0 & 0 & 2 & - 4 & 0 & 0 & 0 & 0 & 0 & 0\\
0 &0& 0 & 2 & 0 & - 4 & 2 & 0 & 0 & 0 & 0\\
0 & 0& 0 & 0 & 0 & 2 & - 4 & 0 & 0 & 0 & 0\\
0 & 0 &0 & 2 & 0 & 0 & 0 & - 4 & 2 & 0 & 0\\
0 & 0 & 0&0 & 0 & 0 & 0 & 2 & - 4 & 2 & 0\\
0 & 0 & 0 &0& 0 & 0 & 0 & 0 & 2 & - 4 & 2\\
0 & 0 & 0 & 0&0 & 0 & 0 & 0 & 0 & 2 & - 4
\end{array} \right)}.$$
Consider the natural involution $$\id_w\oplus \id_{\rU(-1)}\oplus -\id_{E_8(-2)}\in \rS\rO(\Lambda),$$
due to \cite[Proposition 2.2]{vS07}, the action of a symplectic involution on the extended N\'eron-Severi lattice of a $K3$ surface arise in this way. By a direct computation,  we have  
\begin{equation}
\id_\omega\oplus \id_{\rU(-1)}\oplus -\id_{E_8(-2)}=\prod \limits_{i=1}^{10} s_{v_i},
\end{equation}
where  $v_i\in \rU(-1)\oplus \rE_8(-2)\subseteq \Lambda$ are given by 
\begin{equation}
  \begin{aligned}
    &v_1=(1, -1, 3, 2, 2, 1, 2, 2, 2, 1), v_2=(1, -1, 4, 2, 3, 1, 3, 2, 2, 1), \\&  v_3=(1, -1, 0, 0, 0, 0, 0, 0, 1, 1), v_4=(1, -1, 0, 0, 0, 0, 0, 0, 1, 0), \\ & v_5=(1, -1, -1, 0, -1, -1, -1, -1, 0, 0),v_6=(1, -1, -2, -1, -1, -1, -2, -1, 0, 0),\\ 
   & v_7=(1, -1, 0, 0, 0, 0, -1, -1, 0, 0),v_8=(1, -1, -1, 0, 0, 0, -1,-1, 0, 0), \\ &v_9=(3, -3, 1, 1, 1, 0, 0, 0, 2, 1), v_{10}=(1, -1, 0, 0, 0, 0, 0, 0, 0, 0).  
\end{aligned}  
\end{equation}
\end{example}

In general, the condition $$q\not\sim x_1x_2, x_1x_2+x_3x_4 \mod 2$$  in Theorem \ref{thm:Kneser2} and Corollary \ref{cor:ref} is necessary. See the counterexamples below where Corollary \ref{cor:ref} fails. 

\begin{example}[$2$-adic obstructions]

\begin{enumerate}\label{exm:counter}
    \item Let $\Lambda=\rU\oplus \Sigma(2)$ for some even lattice $\Sigma$. Take a non-zero primitive vector $b\in \Sigma(2)$ and  consider the transformation $\rE_b\in \rO(\Lambda)$ defined by 
        \begin{equation}
    \rE_b(re+sf+z)=(-(b, z)+r-q(b)s)e+sf+z+sb,
\end{equation}
where $\{e,f\}$ is the standard basis  of $\rU$ giving \eqref{eq:hyp-basis}.
    It is not hard to see that $\psi$ is an element in $\widetilde{\rS\rO}^\dagger(\Lambda)$ (See Lemma \ref{lem:rel-involution}). Then we claim  $\rE_b$ is not lying in $\rR^\pm(\Lambda)$.  This is because  for any reflection $s_v$ with $q(v)=\pm 1$, we have  
    $$s_v(e-f)\equiv e-f\mod 2\Lambda,$$
    while $\rE_b(e-f)\equiv e-f+b \mod 2\Lambda$.

    \item Let $\Lambda=\rU^{\oplus 2}\oplus \Sigma(2)$ for some even lattice $\Sigma$.   Let $\iota\in \rS\rO(\Lambda)$ be the involution which switches the two hyperbolic planes.  Clearly, we have  $\iota \in  \widetilde{\rS\rO}(\Lambda)$ with spinor norm $-1$. Let $\{e_1,f_1,e_2,f_2\}$ be the standard basis of $\rU^{\oplus 2}$. After modulo everything by $2$, one can see that 
    $\bar{\iota} \notin \rR(\Lambda/2\Lambda)$ (\cf.~\cite[4.2]{Wall63}).  Then  $ s_{e_1-f_1}\circ s_{e_1+f_1}\circ \iota$ is contained in $ \widetilde{\rS\rO}^{\dagger}(\Lambda)$ but not in $\rR^{\pm}(\Lambda)$.

\end{enumerate}

\end{example}

\subsection{Reflective involutions on extended N\'eron-Severi lattice}

For our purpose, we only need to consider  the case  where $\Lambda$ contains a hyperbolic plane.  In this situation, besides the reflections, there is a rich class of reflective involutions.  Such involution was hidden in  the work of Eichler \cite{Ei52} and Wall \cite{Wall63}. They  play  an important role in the study of Hodge conjecture for rational Hodge isometries on $K3$ surfaces and abelian surfaces (\cf.~\cite{B19, Huy19}).

\begin{definition}\label{def:h-inv}
Let $\Lambda=\rU\oplus \Sigma$ be an integral  lattice over $\ZZ$ or $\ZZ_p$. For any $b\in \Sigma$ with $q(b)=n\in \ZZ$ or $\ZZ_p$,  we define $\psi_b:\Lambda\to \Lambda$ to be the map  given by
\begin{equation}
    \psi_b(re+sf+z)=((b, z)-r+q(b)s)e-sf+z+sb,
\end{equation}for $r,s\in A, z\in \Sigma$, where  $\{e,f\}$ is the standard basis of $\rU$. Moreover, we set
\begin{equation}
    \psi_{b,a}=s_a\circ \psi_b\circ s_a
\end{equation}to be the conjugation of $\psi_b$ by a reflection  $s_a$ with $a\in \Lambda$ and $q(a)=\pm 1$.  We may call $\psi_{b,a}$ a  {\it Huybrechts involution}., which is motivated by Huybrechts' work in \cite[\S 1.2]{Huy19} (see \S \ref{sec:tde} for more details). 
\end{definition}

\begin{example}\label{exm:prop-invo}Let us give some basic properties of Huybrecths involutions which will be frequently used in the later computation. 

\begin{enumerate}
\item When $\Lambda=\widetilde{\NS}(X)$ for some K3 surface $X$, a Huybrechts involution $\psi_{b,a}$ is a reflective involution. 
\item When $b=0$, $\psi_0$ is the product $s_{e-f}\circ s_{e+f}$. 
\item (\textbf{Eichler  transvection}) The product $\rE_b:=\psi_0\circ \psi_b$ is the Eichler  transvection, introduced in \cite{Ei52} (see also \cite{Wall63}, \cite{GHS09}).  It is given by  
\begin{equation}
    \rE_b(re+sf+z)= (-(b, z)+r-q(b)s)e+sf+z+sb.
\end{equation}
 Such transformation has occurred in  Example \ref{exm:counter} (1). Note that such transformation is lying in $\widetilde{\rS\rO}^\dagger(\Lambda)$ and it is additive, i.e. $\rE_b\circ \rE_{b'}=\rE_{b+b'}$ for $b,b'\in \Sigma$.  
    \item If $a=e-f$,  we have \begin{equation} 
    \psi_{b,e-f}(re+sf+z)= ((b, z)-s+q(b)r)f-re+z+rb,\end{equation}
    which sends $e$ to $-e+b+q(b)f$ and  $f$ to $-f$.

   % \item The involution $\widetilde{\psi}_b$ defined by  \begin{equation}  \widetilde{\psi}_b(re+sf+z)= ((b, z)+r+q(b)s)e+sf+z+sb, \end{equation}  is the product $ \psi_b \circ s_{e-f}\circ s_{e+f}$. 
     \item If $a=ne+mf\in \rU$ with $q(a)=nm\in \ZZ_p^\times$, then we have 
\begin{equation}\label{eq:conjugation}
\begin{aligned}
s_a\circ \psi_b \circ s_a(re+sf+z)& =    ((-\frac{mb}{n},z)-s +\frac{m^2}{n^2}q(b) r)f- re+ z-r\frac{mb}{n} \\
&= \psi_{-\frac{mb}{n}, e-f}(re+sf+z),
\end{aligned}
\end{equation}
remaining a Huybrechts involution.  
\end{enumerate}

\end{example}

A key observation is 
\begin{lemma}\label{lem:rel-involution}
Let $\Lambda$ be an even lattice over $\ZZ$. Then  $\psi_b\in \rS\rO(\Lambda)$ is a reflective involution with spinor norm $-1$. 
 Moreover, if $q(b)=2k+1$ is odd, then  there exist $v_1,v_2\in \Lambda$ with  $v_1^2=-v_2^2=-2$ such that  $\psi_b=s_{v_1}\circ s_{v_2}\in \rO(\Lambda)$.    
\end{lemma}

\begin{proof}
By a direct computation, we have $\psi_b(e)=-e$ and $\psi_b(f)=q(b)e-f+b$ and $\psi_b(z)=z+(b,z)e$. It is easy to see $\psi_b$ is an involution with spinor norm $-1$. If we take $v=e$ with $v^2=0$ , then we have  $\psi_b(v)=-v$ and 
$$\widetilde{\psi}_b:v^\perp/v \to v^\perp/v, $$ is the identity.  When $q(b)=2k+1$ is odd, take  $v_1=-(k+1)e+2f+b$ and $v_2=-ke+2f+b$.  A short computation shows that
$\psi_b=s_{v_1}\circ s_{v_2}$.    
\end{proof}

We define $\rW(\Lambda)\subseteq \rO(\Lambda)$ to be the subgroup generated by $\rR(\Lambda)$ and $\psi_{b,a}$.  Let $\rW^\pm(\Lambda)\subseteq \rS\rO(\Lambda)$ be the subgroup generated by $\rR^{\pm}(\Lambda)$ and even product of $\psi_{b,a}$. Or equivalent, $\rW^\pm(\Lambda)$ is generated by  $\rR^{\pm}(\Lambda)$ and $s_a\circ \rE_b\circ s_a$.   Every element in $\rW(\Lambda)$ can be written as 
\begin{equation}
    \prod\limits_{i=1}^k \psi_{b_i} \circ \prod\limits_{j=1}^l s_{v_j}, 
\end{equation}
with $b_i\in \Sigma$ and $v_j\in \Lambda$. 

%From the construction, the involution $\widetilde{\psi}_b$ is also lying in $\rW^\pm(\Lambda)$. 

\subsection{Group generated by the Huybrechts involutions}  Now we are going to prove our main result. 
\begin{theorem}\label{thm:kne3}
   Let $\Sigma$ be a non-degenerate even lattice  over $\ZZ$  and let $\Lambda=\rU\oplus \Sigma$. Then 
    \begin{enumerate}
        \item [$(1')$] if $\Sigma$ is indefinite,   $\rW^\pm(\Lambda)=\bigcap\limits_p (\rW^\pm(\Lambda_p) \cap \rS\rO(\Lambda)) $ when $\rW^\pm(\Lambda)$ is a congruence subgroup;
        \item [$(2')$] $\rW(\Lambda_p)=\widetilde{\rO}(\Lambda_p)$ for all $p$;
        \item [$(3')$] $\rW^\pm(\Lambda_p)=\widetilde{\rS\rO}^{\dagger}(\Lambda_p)$.
    \end{enumerate}
\end{theorem}
\begin{proof} The idea is similar to  Theorem \ref{thm:Kneser2}, but the proof is more subtle.  We use $q_\Sigma$ to denote the  quadratic form on $\Sigma$.

\subsubsection*{(1')}   Under our assumption, $\rW^\pm(\Lambda)$ contains a principal congruence subgroup $\Gamma(m)$ of level $m$.  As before, we let $S$ be the collection of prime factors of $m$.  As $\rW^{\pm}(\Lambda_p)=\rR^\pm(\Lambda_p)$ when $p\neq 2$,  the assertion follows from  the same argument in  Theorem \ref{thm:Kneser2} (1)  if $2 \notin S$. 

Suppose $2\in S$.  For $g\in \bigcap\limits_p (\rW^\pm(\Lambda_p) \cap \rS\rO(\Lambda))$,  by first viewing it as an element in   $\rS\rO(\Lambda_2)$, we can  write $g$  as a product 
\begin{equation}
 g=  \prod\limits_{j=1}^k  \psi_{b_j^{(2)}, a_j^{(2)}}\circ \prod\limits_{i=1}^l  (s_{c^{(2)}_{2i-1}} \circ s_{c^{(2)}_{2i}}),
\end{equation}
 where  $b_i^{(2)}\in \Sigma_2, a_i^{(2)}, c_i^{(2)}\in \Lambda_2$ satisfying  $q(c_{2i-1}^{(2)})=q(c_{2i}^{(2)})=\pm 1$. Then we can find vectors in $b_i\in \Sigma$ and $a_i\in \Lambda$ with $q(a_i)=q(a_i^{(2)})$ such that $b_i\equiv b_i^{(2)}\mod 2$ and $a_i\equiv a_i^{(2)}\mod 2$.  Again, the existence of such $a_i$ is ensured by the strong approximation. 

 Replace $g$ by $ (\prod\limits_{j=1}^k  \psi_{b_j, a_j} )^{-1} \circ g$, as in Theorem \ref{thm:Kneser2} (1), we can find an element $h\in \rR^\pm(\Lambda)$ to approximate $g$ at places in $S$, which means $h^{-1}\circ g$ is lying in the principal congruence subgroup $\Gamma(m)$. It follows  $g\in \rW^\pm(\Lambda)$ as well.

\subsubsection*{(2')} By Theorem \ref{thm:Kneser2} (2),  we only need to treat the case when $p=2$ and $q_\Sigma \mod 2$ is equivalent to $0$ or $x_1x_2$.

 If $q_\Sigma =0 \mod 2$,  for any $g\in \widetilde{\rO}(\Lambda_2)$,  one can first diagonalize $g$  by using reflections and  $\psi_{b,a}$. Indeed, if we write
\begin{equation}
    g(e)=n_1e+m_1f+z_1 ~\hbox{and}~g(f)=n_2e+m_2f+z_2,
\end{equation}
with $n_i, m_i\in \ZZ_2$ and $z_i\in \Sigma_2$,  then we have $n_1$ or $m_1\in \ZZ_2^\times$ and $n_2$ or $m_2\in \ZZ_2^\times$ because  $(g(e), g(f))=-1$ and $(z_1,z_2)\notin \ZZ_2^\times$.  Assume $m_1\in\ZZ_2^\times$, otherwise we can multiply $g$ with the reflection $s_{e-f}$ which switches $e$ and $f$. Then there will be $$q(e+g(e))=(e, g(e))=m_1\in \ZZ_2^\times ,$$ 
and  the product $s_{e+g(e)}\circ g$   sends $e$ to $-e$ and $f$ to $q(b)e-f+b$ for some $b\in \Sigma_2$.   Hence $(\psi_b^{-1} \circ s_{e+g(e)}\circ g) |_{\rU}=\id$.  

Replace $g$ by its diagonalization  $\psi_b^{-1} \circ s_{e+g(e)}\circ g $ and consider the sublattice 
\begin{equation}\label{eq:diag}
    \widehat{\Lambda}:=\ZZ (e+f)\oplus \Sigma \subseteq \Lambda.
\end{equation} 
Then  $q_{\widehat{\Lambda}_2}\mod 2$ is equivalent to  $x_1^2$. Note that the restriction $g|_{\widehat{\Lambda}_2}$ also acts trivially on the discriminant group of $\widehat{\Lambda}_2$,  we have  $g|_{\widehat{\Lambda}} \in \rR(\widehat{\Lambda}_2) $
by Theorem \ref{thm:Kneser2} (2).  Since $g$ acts as identity on the orthogonal complement of $\widehat{\Lambda}_2$ in $\Lambda_2$, we have $g\in \rR(\Lambda_2)$ as well. 

If $q_\Sigma\mod 2$ is equivalent to $x_1x_2$,  then $\Sigma_2\cong \rU\oplus \Omega(2)$, where  $\Omega$ is a $\ZZ_2$-lattice.  Again, let us first show any  $g\in\widetilde{\rO}(\Lambda)$ can be diagonalized.   If $g(e)=-e$,  there must be 
$$g(f)=-f+q(b)e+b,$$ for some $b\in \Sigma_2$. Then the restriction $(\psi_b^{-1}\circ g )|_{\rU}$  is the identity.  So the problem becomes  to find $h\in \rW(\Lambda_2)$ such that $h\circ g(e)=-e$.  Let $\{e',f'\}$ be the standard basis of $\rU$ and we set 
 $$g(e)=ne+mf+(n'e'+m'f')+w,$$ for some $w\in \Omega(2)$.  If $m$ or $n\in \ZZ_2^\times$, we can take $h=s_{e+g(e)}$ or $s_{e+f}\circ s_{f-g(e)}$ respectively . If $n, m\notin \ZZ_2^\times$, one must have $n'$ or $m' \in\ZZ_2^\times $ since $(g(e),g(f))=1$.  Without loss of generality, we can assume $n'\in \ZZ_2^\times$.  Then  we get
 $$\psi_{e'}\circ s_{f'-g(e)}\circ g(e)=\psi_{e'}(f')=f'+e.$$ 
As $q(f-e-f')=-1\in \ZZ_2^\times$,  we can take $h=s_{e+f}\circ s_{f-e-f'}\circ \psi_{e'}\circ s_{f'-g(e)}$ as desired.

Now suppose $g|_{\rU}=\id$.  Consider the sublattice $\widehat{\Lambda}\subseteq \Lambda$ defined  in \eqref{eq:diag},  we have $q_{\widehat{\Lambda}}\mod 2$ is equivalent to $x_1^2+x_2x_3$. Then the restriction $g|_{\widehat{\Lambda}}$ is contained in $\rR(\widehat{\Lambda}_2)$ which also deduces  $g\in \rR(\Lambda_2)$.

\subsubsection*{(3').} Due to $(1')$, $(2')$ and Theorem \ref{thm:Kneser2} (3), it suffices to consider the case $q_\Sigma =0 \mod 2$ and show that the product  $s_a\circ s_c, a,c\in\Lambda_2$ with $q(a)=q(c)=\pm 3$  is contained in $\rW^\pm(\Lambda_2)$. In this case, $s_a\circ s_c\in \rR^{\pm}(\Lambda_2)\subseteq \rW^{\pm}(\Lambda_2)$ by Theorem \ref{thm:Kneser2} (3) since $q_{\Lambda}\mod 2$ is equivalent to $x_1 x_2$.

\end{proof}

\begin{remark}
  We find that  there is a similar result in  \cite[Theorem 7.9]{MM09}  for Theorem \ref{thm:kne3} (2') by using the generalized Eichler transvections.
\end{remark}

This  immediately yields
\begin{corollary}\label{cor:W=O}
If the Witt index of $\Lambda_\RR$ is at least $2$ and $\rank(\Sigma)\geq 3$,  then $\rW(\Lambda)=\widetilde{\rO}(\Lambda)$. 
\end{corollary}
\begin{proof}
 The proof is similar as  Corollary \ref{cor:ref}. Under our assumption, $\rW^\pm(\Lambda)$ is a congruence subgroup of  $\rS\rO(\Lambda)$.  By Theorem \ref{thm:kne3}, we have  $\rW^\pm(\Lambda)= \widetilde{\rS\rO}^\dagger(\Lambda)$. This also implies $\rW(\Lambda)=\widetilde{\rO}(\Lambda)$ as the coset $\widetilde{\rO}(\Lambda)/\widetilde{\rS\rO}^\dagger(\Lambda)$ is generated by $\pm 2$ reflections. 
\end{proof}

\begin{remark}\label{rmk:pic1}
When $\rank~\Sigma=1$,  we still have $\widetilde{\rS\rO}^{\dagger}(\Lambda_p)= \rW^{\pm}(\Lambda_p)$ for any $p$. However, as $\rW^{\pm}(\Lambda)$ may not be congruence, the assertion in Corollary \ref{cor:W=O} fails in general.  See the counterexample in \S \ref{sec:pic1}. 
\end{remark}

\subsection{Proof of Theorem \ref{thm:thm2}}
By Theorem \ref{thm1},  as $\Phi^{\widetilde{\NS}}$ is lying in the stable orthogonal group $\widetilde{\rO}(\widetilde{\NS}(X))$,  we only need to show that $\widetilde{\rO}(\widetilde{\NS}(X))$ is generated by  reflective involutions. 

We can apply Corollary  \ref{cor:W=O} and Proposition \ref{prop:unigonal}  to the lattice $\widetilde{\NS}(X)(-1)$. They ensure  that $\widetilde{\rO}(\widetilde{\NS}(X))$ is generated by the Huybrechts involutions and $\pm 2$ reflections.  As mentioned in Example \ref{exm:prop-invo} (1), any Huybrechts involution is a reflective involution. The assertion then follows.

%% file: sec7.tex
\section{Bloch's conjecture for  $K3$ surfaces with small Picard number }\label{sec:pic12}
In this section, we will focus on the case  $\rank (\NS(X))\leq 2$ and provide examples where the identity $\rR^\pm(\Lambda)=\widetilde{\rS\rO}^\dagger(\Lambda)$ in Corollary \ref{cor:ref} remains valid.  Moreover, we use Huybrechts' autoequivalence to outline an approach to Bloch's conjecture for symplectic automorphisms on $K3$ surfaces.

\subsection{$K3$ surface of  Picard number one}\label{sec:pic1}
If $\rank (\NS(X))=1$,  we can set $\NS(X)=\ZZ L$ with $L^2=2n$. Then  we have
\begin{enumerate}
    \item $\Aut(X)\cong 0$ or $\ZZ/2\ZZ$. 
    \item Every autoequivalence is symplectic or anti-symplectic. 
    
\end{enumerate}

In this case,  the question of whether $\rW(\widetilde{\NS}(X))=\widetilde{\rO}(\widetilde{\NS}(X))$ becomes quite subtle.  We  will focus on case when $n$ is small.  
First of all, we have 
\begin{lemma}
Set $\mathbf{L_n}=\widetilde{\NS}(X)(-1)$. Then 
    $\rR^\pm(\mathbf{L_n})=\rW^\pm(\mathbf{L_n})$ if and only if $n$ is odd. 
\end{lemma}
\begin{proof}
 Due to Lemma \ref{lem:rel-involution}, Example \ref{exm:counter} and Example \ref{exm:prop-invo}, we have 
\begin{itemize}
    \item  $\rE_{L}=\psi_L\circ s_{e-f}\circ s_{e+f}\in \rR^\pm(\mathbf{L_n})$ if and only if $n$ is odd. 
    \item  $\rE_{mL}=(\rE_L)^m$.
\end{itemize}
As $\rW^\pm(\mathbf{L_n})$ is generated by   $\rR^\pm(\mathbf{L_n})$ and conjugates of $\rE_{mL}$, this proves the first assertion.
\end{proof}

For hyperbolic lattices of rank $>2$, it is well-known that there are only finitely many isomorphism classes $\Lambda$ such that the Weyl group  $\rR^-(\Lambda)$ which by definition is generated by the reflections $s_v$ with $v^2=-2$ has finite index in $\rO(\Lambda)$ (\cf.~\cite{Nik83}, \cite{Nik85}, \cite{Vin07}).  When   $\Lambda=\mathbf{L_n}$, this happens if and only if $n=1$.  In the following, let us compute the index of $\rW^\pm(\mathbf{L_n})$ in  $\rS\rO^{\dagger}(\mathbf{L_n})$ for small values of $n$.

\begin{proposition}\label{prop:picard1}
    Let $\mathcal{N}$ be the set of integers 
    \[ \{ 1,2,\ldots 11, 13,\ldots, 17, 19, 21, 23, 25, 26, 29, 31, 41, 49 \}.\] Then $\widetilde{\rO}(\mathbf{L_n})=\rW(\mathbf{L_n})$ for all $n\in \mathcal{N}$. 
\end{proposition}
\begin{proof}
  When $n=1$-$5, 7, 9, 13, 25$,  this was already shown in  \cite{Nik99}.  For the remaining values in $\mathcal{N}$, one can  get a finite set of isometries $g_i\in \widetilde{\rO}(\mathbf{L_n})$ which generate the quotient group $\widetilde{\rO}^+(\mathbf{L_n})/ \rR^-(\mathbf{L_n})$ by using Shimada's algorithm in \cite{Shi15}.  By adapting the reduction method of Wall in \cite[Section 5]{Wall63}, we can find explicit decomposition of $g_i$ as product of reflective involutions, which implies $\widetilde{\rO}(\mathbf{L_n})=\rW(\mathbf{L_n})$ since $s_{e+f}\in \rW(\mathbf{L_n})$ generates $\widetilde{\rO}(\mathbf{L_n})/ \widetilde{\rO}^+(\mathbf{L_n})$. 
  
Let us treat the case when  $n=10$ while the computations of other cases are similar.
After running  the computer program via Shimada's algorithm \cite{Shi15}, we can  obtain a set $\{g_1, g_2\}$ of  generators of  $g_i\in \widetilde{\rO}(\mathbf{L_n})$, where $g_1=\rE_{L}$ is an Eichler  transvection,  $g_2$ is defined by the gram matrix  $$g_2={\left(\begin{array}{ccc} 
81 & 160 & 720 \\
40 & 81 & 360 \\
18 & 36 & 161
\end{array} \right)} .$$   
Note that $g_2(e)=81 e+40 f+18 L$. 
Then Wall's reduction method is to decrease the absolute value of the coefficients of $e$ and $f$  in $g_2(e)$ via  composing with reflections around vectors of square $\pm 2$. In our situation, by searching  $\pm 2$ vectors with coordinates bounded by a reasonably large number, we find $v_1=13e+7f+3L$ and 
$$s_{v_1}\circ g_2(e)=-10 e-9f-3L.$$
Repeating this process, we get $v_2=3e+3f+L$ with  $$s_{v_2}\circ s_{v_1}\circ g_2(e)=s_{v_2}(-10 e-9f-3L)=- e.$$  Then $s_{v_2}\circ s_{v_1}\circ g_2$ is a Huybrechts involution, which shows $g_2\in \rW(\mathbf{L_n})$.  See \cite[page 1-3]{link} for the computation of other values of $n$. \end{proof}

Unfortunately,  there do exist infinitely many integers  $n$ such that  $$\rW^{\pm}(\mathbf{L_n})\neq \widetilde{\rS\rO}^{\dagger}(\mathbf{L_n}),$$
is not a congruence subgroup.  This is due to the following observation
\begin{proposition}\label{lem:indexinf}
The index $[\rO(\mathbf{L_n}): \rR(\mathbf{L_n})]=\infty$  if the following conditions hold:
\begin{enumerate}
    \item for each $v\in \mathbf{L_n} $ with $v^2=2$, $v^\perp$ is isomorphic to  $\left< -2\right>\oplus \left< -2n\right>$. 

    \item $\mathbf{L_n}$ is not reflective of elliptic type in the sense of Nikulin \cite{Nik99}. 
\end{enumerate}
 As a consequence,  we have  $[\rO(\mathbf{L_n}): \rR (\mathbf{L_n})]=\infty$ when $n=37d^2$ for any $d\geq 1$. 
\end{proposition}
\begin{proof}
 For the first assertion, let $G_1$ be the subgroup generated by $\rR(\mathbf{L_n})$ and $-\id$. Set $G_2=G_1\cap \rO^+(\mathbf{L_n})$.  Then $G_2$ is generated by the reflections $s_w$ with $w^2=-2$ and the involutions $-s_{v}$ with $v^2=2$.  Since $\rO^+(\mathbf{L_n})$ has finite index in $\rO(\mathbf{L_n})$,  it suffices to show $[\rO^+(\mathbf{L_n}):G_2]=\infty$. We define $$W\subseteq \rO(\mathbf{L_n})$$  to be the subgroup generated by the integral reflections $s_w$ with  $w^2<0$. By Condition (1),  for every root $v\in \mathbf{L_n}$,  there exist  $w_1,w_2\in v^\perp$ such that $w_1^2=-2$, $w_2^2<0$ and  $$-s_{v_1}=s_{w_1} s_{w_2}.$$ This implies  that  $G_2\subseteq W$. As Condition (2) exactly means $[\rO^+(\mathbf{L_n}):W]=\infty$, we obtain the assertion.

For the second assertion,  we first consider the case $n=37$.  As shown in  \cite{Nik99}, $\mathbf{L_{37}}$ satisfies the condition (2).  Moreover, for any vector  $v\in \mathbf{L_{37}}$ with $v^2=2$,  as $v\oplus v^\perp\neq \widetilde{\NS}(X)$,  $v^\perp$ has to be a rank $2$ negative definite even lattice of determinant $148$. By classification of binary positive definite quadratic forms, $ v^\perp$ is isometric to either $N_1:={\left(\begin{array}{cc} 
-2 & 0 \\
0 & -74 
\end{array} \right)}$ or $N_2:={\left(\begin{array}{cc} 
-4 & 2 \\
2 & -38 
\end{array} \right)}$. However, $\left< 2\right>\oplus N_2$ can not be embedd into $\mathbf{L_{37}}$.  Thus, $v^\perp \cong N_1$.  This shows  $[\rO(\mathbf{L_n}): \rR(\mathbf{L_n})]=\infty$ when $n=37$.  

For $n=37d^2$, there is a natural lattice embedding
\begin{equation}
   \mathbf{L_n}\hookrightarrow  \mathbf{L_{37}}. 
\end{equation}
This induces natural inclusions
$\rR(\mathbf{L_n})\subseteq \rR(\mathbf{L_{37}})$  and $\widetilde{\rO}(\mathbf{L_{n}})\subseteq \rO(\mathbf{L_{37}})$ with $[\rO(\mathbf{L_{37}}):\widetilde{\rO}(\mathbf{L_{n}})]<\infty$. 
As $[\rO(\mathbf{L_{37}}):\rR(\mathbf{L_{37}})]=\infty$, 
it follows that  $[\rO(\mathbf{L_n}):\rR(\mathbf{L_n})]=\infty$ as well. 

\end{proof}

\begin{remark}
There are only finitely many $n$ where  condition (1) holds. On the other hand, condition $(2)$ holds for almost all positive $n\in \ZZ$  (\cf.~ \cite{Nik00}, \cite{All12}). Note that among the integers that appear in Proposition \ref{prop:picard1}, condition $(2)$ does hold for the four integers $n=23, 29,31, 41$. It will be  a very interesting question whether $\rW^{\pm}(\mathbf{L_n})= \widetilde{\rS\rO}^{\dagger}(\mathbf{L_n})$ for infinitely many $n$. 

\end{remark}

\subsection{$K3$ surface of Picard number two}
In the case $\rank~ \NS(X)=2$, we can set $$\NS(X)=\Big(\begin{array}{cc}
  2a   & b \\
   b  & 2c
\end{array}\Big).$$  If $\NS(X)\cong \rU$, Wall has shown in \cite[Theorem 6.12]{Wall63} that 
\begin{equation}\label{eq:wall}
    \rW^\pm(\rU\oplus \rU)=\rS\rO^{\dagger}(\rU\oplus \rU). 
\end{equation}
This yields 
\begin{proposition}\label{prop:unigonal}
  If $X$ is admits a Jacobian fibration,  then  Conjecture \ref{conj2} holds.
\end{proposition}
\begin{proof}
Recall that $X$ is admits a Jacobian fibration if and only if  $\NS(X)$ contains a hyperbolic lattice $\rU$. The assertion follows from Theorem \ref{thm1} and Wall's result.  
    
\end{proof}

When $\NS(X)$ is not unimodular, we currently do not have a counterexample where $\rW(\NS(X))$ is not a congruence subgroup. Alternatively, we may consider the automorphism group of $X$ which is well-understood via its N\'eron-Severi lattice.  
It has been shown in \cite{GLP10} that the automorphism group $\Aut(X)$ is   one of the three cases
\begin{enumerate}
    \item $|\Aut(X)|<\infty$ 
     \item $\Aut(X)\cong \DD_\infty$ is an infinite dihedral group.
     \item $\Aut(X)\cong \ZZ$  is an infinite cyclic group. 
\end{enumerate}
 In case (1) and (2),  Bloch's conjecture holds for $\Aut^s(X)$ because $\Aut^s(X)$ is either trivial or generated by anti-symplectic involutions whose action on $\CH_0(X)_{\hom}$ is $-\id$. The challenging problem is case (3), where the N\'eron-Severi lattice of $X$ can be classified as below. 
\begin{lemma}
   $\Aut(X)\cong \ZZ$ if and only if $\NS(X)$ does not contain vectors of square $0$, $\pm 2$.   
\end{lemma}

\begin{proof}
 It is well-known that   $\Aut(X)$ is a finite group if and only if $\NS(X)$ contains an isotropic vector or a root.  When $\Aut(X)$ is infinite, it contains an anti-symplectic involution if and only if $\NS(X)$ contains an ample line bundle of degree $2$. 
\end{proof}

Set $D=(b^2-4ac)>0.$ Then one has

\begin{theorem}\cite[Main Theorem]{HKL20}\label{thm:aut-pic2}
An element $g\in \Aut(X)$ is symplectic of infinite order if  and only if its action on Picard lattice $g^\ast |_{\NS}:\NS(X)\to \NS(X)$ is of the form 
\begin{equation}\label{eq:sym}
g_{u,v}= \left(\begin{array}{cc}
    \frac{u-bv}{2}  &-cv  \\
    av  & \frac{u+bv}{2}
 \end{array}\right),
\end{equation} 
where  $u=\alpha^2-2$, $v=\alpha\beta$ and $(\alpha, \beta) $ is a positive solution of the Pell equation $x^2-Dy^2 =4$ with $\alpha\geq 4$.
\end{theorem}

Due to Theorem \ref{thm:aut-pic2},  we are motivated to consider a concrete question:  when is the diagonal transformation $$\id_{\rU}\oplus g_{u,v} \in \rS\rO(\widetilde{\NS}(X))$$  lying in $\rW(\widetilde{\NS}(X))$? 
A positive answer leads to a solution of Bloch's conjecture for $\Aut^s(X)$. Though we suspect it will fail in general, a quick computer search (see \cite[page 4-143]{link}) gives the following result  

\begin{proposition}
Bloch's conjecture holds for  $\phi\in \Aut^s(X)$ if $\NS(X)\cong \Sigma(d)$ for some even lattice $\Sigma$ with $|\det(\Sigma)|\leq 100$ and $|d|\leq 6$ except $\Sigma=  \left(\begin{matrix}
    -2 & 9 \\
    9 & 4
    \end{matrix}\right)$, $d=\pm 3, \pm 6$ or  $\Sigma=    \left(\begin{matrix}
    4 & 7 \\
    7  & -4
\end{matrix}\right) $,   $    \left(\begin{matrix}
    -2 & 9 \\
    9  & 8
\end{matrix}\right) $,  $    \left(\begin{matrix}
    -2 & 7 \\
    7  & 8
\end{matrix}\right) $ and  $d=\pm 5$.

\end{proposition}

\begin{example}A typical example is from Oguiso (\cf.~\cite[Theorem 6.6]{Ogu14}). Suppose $$\NS(X)= \left(\begin{array}{cc}
    4  & 2  \\
    2  & -4 
 \end{array}\right)
$$ 
with basis $L_1$ and $L_2$.
Then $\Aut(X)\cong \ZZ$ is generated by an anti-symplectic automorphism $g$ and $g^{\widetilde\rH}=-s_{a_1}\circ s_{a_2}\circ s_{a_3}\circ s_{a_4}$, where $a_1=(1,0,-1)$, $a_2=(1,-L_1,1)$, $a_3=(1,-L_1-L_2,1)$ and $a_4=(1,-2L_1-3L_2,1)$.
\end{example}

\subsection{A reduction of Bloch's conjecture for $\Aut^s(X)$}\label{sec:tde}

At last, we discuss a possible generalization of Theorem \ref{thm:aut-inj}.
\begin{definition}
 A twisted $K3$ surface $\srX\to X$  is {\it essentially trivial} if the associated Brauer class $[\srX]$ is zero in $\Br(X)$.    
\end{definition}
 
In this case, the induced long exact sequence of \eqref{eq:Kummer}  becomes  $$0\rightarrow \Pic(X)/m\Pic(X)\xrightarrow{\delta} \rH^2(X,\bm\mu_m)\rightarrow {\rm Br}(X)[m]\rightarrow 0.$$  As $[\srX]=0$ in $\Br(X)[m]$, there exists $L\in\Pic(X)$ such that $\srX\to X$ corresponds to  $(X,\delta(L))$. One can obtain  a $B$-field lift $B=\frac{c_1(L)}{m}\in \rH^2(X,\QQ)$ of $\delta(L)$. Moreover, there is a derived equivalence 
\begin{equation}\label{eq:etde}
 L^{\frac{1}{m}}\otimes (\_):   \rD^b(X)\cong\rD^{(1)}(\srX),
\end{equation} given by tensoring the twisted line bundle $L^{\frac{1}{m}}$.  The  induced action of \eqref{eq:etde} on  $\CH^*(X)_\QQ$  is given by 
\begin{equation} \label{esstrivial}
\begin{aligned}
    {\bf exp}({\frac{c_1(L)}{m}}): \CH^*(X)_\QQ &\longrightarrow \CH^*(X)_\QQ\\
    \ch(E)&\mapsto \ch_\srX(E\otimes L^{\frac{1}{m}}).
\end{aligned}
\end{equation}
 Note that the action on $\CH_0(X)_{\rm hom}\otimes\QQ$ is the identity.
Then Bloch's conjecture predicts

\begin{conjecture}\label{conj:tde}
Let $\Aut^{\rm tde}(\CH^\ast(X)_\QQ)\subseteq \Aut(\CH^\ast(X)_\QQ)$  be the 
subgroup generated by actions induced by derived autoequivalences and ${\bf exp}({\frac{c_1(L)}{m}})$ for $L\in \Pic(X)$ and $m\in \ZZ.$  The induced map \begin{equation}\label{tdeinj}
    \Aut^{\rm tde}(\CH^\ast(X)_\QQ)\rightarrow \Aut(\widetilde\rH(X,\QQ)).
\end{equation}  is injective. 
\end{conjecture}

When the autoequivalence $\Phi$ acts diagonally on $$\widetilde{\rH}(X)=\rU\oplus \rH^2(X,\ZZ),$$ there is another way to decompose the action of $\Phi^{\CH}$. According to the Cartan-Dieudonne decomposition,  the action of $\Phi$ on   $\NS(X)$ can be composed into a product of rational reflections:  $$\Phi^{\NS}=\prod \limits_{i=1}^{m} s_{b_i},$$  for some vectors  $b_i\in \NS(X)$.
   If $\Phi$ is symplectic, the decomposition can be further extended to $\rH^2(X,\QQ)$ as rational Hodge isometries. Moreover, we have

\begin{lemma}\label{lem:CDdec}
 If the spinor norm of $\prod\limits_{i=1}^{2n} s_{b_i}$ in $\QQ^\times/(\QQ^{\times })^2$ is $\pm 1$, then there exist vectors $$c_1,\ldots,c_{2n}\in \NS(X)$$ such that  
 \begin{equation}
 \prod\limits_{i=1}^{2n} s_{b_i}=\prod\limits_i^{2n} s_{c_i},
 \end{equation}
 and \begin{equation}\label{eq:squreeq}
     \prod_{i=1}^n q(c_{2i-1})=\pm \prod_{i=1}^n q(c_{2i}).
 \end{equation}
\end{lemma}
\begin{proof}
By definition,   we have 
$\prod\limits_{i=1}^{2n} q(b_{i})=\pm k^2$ for some $k\in \ZZ$.  Set $m=\prod\limits_{i=1}^n q(b_{2i})$, then 
$$\prod_{i=1}^n \frac{q(b_{2i-1})}{q(b_{2i})}=\pm \frac{k^2}{m^2}.$$
Then we can take $c_1=mb_1$, $c_2=kb_2$ and $c_i=b_i$ for $i\geq 3$ which satisfies the condition.  

\end{proof}

 For any $b\in \NS(X)$, the reflection $s_b$  defines  a rational Hodge isometry $$s_b: \rH^2(X,\QQ)\cong \rH^2(X,\QQ).$$  
Let ${B}=\frac{2{b}}{({b},{b})}\in \NS(X)_\QQ$. Then $s_b$  can be reviewed as  an integral  reflection on the sublattice  $$\rH^2(X,\ZZ)_{B}:=\{u\in \rH^2(X,\ZZ)~|~(u, B)\in \ZZ\}.$$
In \cite[\S 1.2]{Huy19},   Huybrechts showed that such reflection $s_{b}$ can be lifted to an integral  Hodge (since $B\in\NS(X)_\QQ$) isometry  $\psi_b:\widetilde{\rH}(X)\to  \widetilde{\rH}(X)$
 via the diagram of Hodge structure morphisms \begin{equation}
\begin{tikzcd}\label{huyinv}
{\rH^2(X,\ZZ)_{B}} \arrow[d, "{\bf exp}(B)", hook] \arrow[rr, "s_{b}"] &  & {\rH^2(X,\ZZ)_{B}} \arrow[d, "{\bf exp}(B)", hook] \\
\widetilde{\rH}(X) \arrow[rr, "\psi_{b}"]                       &  & \widetilde{\rH}(X)                     \end{tikzcd},
 \end{equation}
 where $\psi_b$ is the involution in Definition \ref{def:h-inv}.  A lifting   $\Psi_b\in \Aut(\rD^b(X))$ of   $\psi_b$ will be a  reflective anti-autoequivalence.  Due to  Lemma \ref{lem:rel-involution} and Theorem \ref{thm:maingeo}, we have 
\begin{equation}\label{eq:h=-id}
    \Psi_b^{\CH}=-\id\in \Aut(\CH_0(X)_{\hom}). 
\end{equation}

Then we can obtain a reduction of Bloch's conjecture for automorphisms.

\begin{proposition}
  For any symplectic  (anti-)autoequivalence $\Phi\in \Aut(\rD^b(X))$ whose action on $\widetilde{\rH}(X)$ is diagonal, the map 
\begin{equation}
       \Phi^\CH:\CH_0(X)_{\hom}\to \CH_0(X)_{\hom}, 
   \end{equation} is $\mathrm{id}$ (resp. $\mathrm{-id}$) if the map \eqref{tdeinj} is injective.   
   \end{proposition}
%whose  action on $\widetilde{\rH}(X)$ is an involution, which sends $\sigma$ to $\sigma$ (resp. $-\sigma$).  Then the action  on zero cycles \begin{proposition}  Any reflective anti-autoequivalence $\Psi_b$ acts on $\CH_0(X)_{\hom}$ as $-\id$.    \end{proposition}

\begin{proof}
The idea is essentially due to the work in  \cite{Huy19}.  By possibly  composing with spherical twists or shifting functors, we may assume that the diagonal action $\Phi^{\widetilde{\rH}}$ is $\id_{\rU}\oplus \Phi^{\NS}\in \rS\rO(\widetilde{\NS}(X))$. 
By using the Cartan-Dieudonne decomposition and Lemma \ref{lem:CDdec},  we obtain  a sequence of derived anti-autoequivalences $$\Psi_{i}:\rD^b(X)\to \rD^b(X)^{\rm op},$$
for $i=1,\ldots, 2m$ with  $2m\leq \rank (\NS(X))$ so that  $\Phi^{\NS}=\prod\limits_{i=1}^{2m}s_{b_i}$.  By \cite[Theorem 6.2]{GMC02}, $\Phi^{\NS}$ is of spinor norm $1$ since it can be extended to an isometry on $\widetilde\rH(X,\ZZ)$. Hence we can choose $b_i\in \NS(X)$ so that $q(b_i)$ satisfies \eqref{eq:squreeq}. 

For $B_i=\frac{b_i}{q(b_i)}\in \NS(X)_\QQ$, let $L_i$ be a line bundle with class $b_i$.
We have essentially trivial derived equivalences $$L_i^{-\frac{1}{q(b_i)}}\otimes (\underline{~~}):\rD^b(X)\rightarrow \rD^{(1)}(\srX_i),$$ as in \eqref{eq:etde}, where $\srX_i\to X$ is a twisted $K3$ surface with a $B$-field $-B_i$. If we let $$\widehat\Psi_{b_i}:=L_i^{-\frac{1}{q(b_i)}}\otimes (\underline{~~})\circ \Psi_{b_i} \circ L_i^{\frac{1}{q(b_i)}}\otimes (\underline{~~}):\rD^{(1)}(\srX_i)\to \rD^{(1)}(\srX_i)^{\rm op},$$
then each derived equivalence $\widehat\Psi_{b_i}$  induces a rational Hodge isometry $$\widehat \Psi_{b_i}^{\widetilde{\rH}}:\widetilde{\rH}(X,\QQ)\to \widetilde{\rH}(X,\QQ).$$
Note that  the restriction $\widehat \Psi_{b_i}^{\rH^2}=s_{b_i}$ on $\rH^2(X,\QQ)$ and  
$\widehat \Psi_{b_i}^{\widetilde{\rH}}$ maps $\rU$ to $\rU$ by 
$$\widehat \Psi_{b_i}^{\widetilde{\rH}}(0,0,1)={\bf exp}(-B_i)\cdot  \Psi^{\widetilde\rH}_{b_i}(0,0,1)={\bf exp}({-B_i})\cdot(-n,-b_i,-1)=(-n,0,0),$$ and 
$$\widehat \Psi_{b_i}^{\widetilde{\rH}}(1,0,0)={\bf exp}(-B_i)\cdot  \Psi^{\widetilde\rH}_{b_i}(1,B,\frac{B_i^2}{2})={\bf exp}({-B_i})\cdot(0,0,-\frac{1}{n})=(0,0,-\frac{1}{n}),$$ where $n=q(b_i).$  Comparing the actions of $\Phi$ and $\prod\limits_{i=1}^{2m} \widehat \Psi_{b_i}$, we have 
    $$\Phi^{\widetilde \rH}=\prod\limits_{i}^{2m} s_{b_i}= \prod\limits_{i=1}^{2m} \widehat \Psi_{b_i}^{\widetilde{\rH}}.$$
    Since $\widehat{\Psi}_{b_i}^{\CH}=-\id$ by \eqref{eq:h=-id}, it follows that $\Phi^{\CH}=\id $ provided \eqref{tdeinj} is injective. 

\end{proof}